\documentclass[12pt, a4paper, reqno, oneside]{amsart}
\usepackage{amssymb}  
\usepackage{amsmath} 
\usepackage{amsthm}  

\usepackage[cp1251]{inputenc}
\usepackage[english]{babel}
\usepackage{cite, verbatim, multirow}

\newtheorem{theorem}{Theorem}

\newtheorem*{cor*}{Corollary}
\newtheorem{lemma}{Lemma}
\numberwithin{lemma}{section}

\numberwithin{equation}{section}

\newcommand{\Aut}{\operatorname{Aut}}
\newcommand{\Out}{\operatorname{Out}}
\newcommand{\Inndiag}{\operatorname{Inndiag}}

\newcommand{\diag}{\operatorname{diag}}

\renewcommand\labelenumi{(\roman{enumi})}
\renewcommand\theenumi\labelenumi

\begin{document}

\title[On orders of elements of finite almost simple groups]{On orders of elements of finite almost \\ simple groups with linear or unitary socle}

\author{M.A. Grechkoseeva}

\address{Sobolev Institute of Mathematics, Ac. Koptyuga 4, Novosibirsk, 630090, Russia;
Novosibirsk State University, Pirogova 4, Novosibirsk, 630090, Russia}
\email{grechkoseeva@gmail.com}

\thanks{The work is supported by  Russian Science Foundation (project 14-21-00065)}

\begin{abstract}
We say that a finite almost simple $G$ with socle $S$ is {\it admissible} (with respect to the spectrum) if $G$ and $S$ have the same sets of orders of elements.
Let $L$ be a finite simple linear or unitary group of dimension at least three over a field of odd characteristic.
We describe admissible almost simple groups with socle $L$. Also we calculate the orders of elements of the coset
$L\tau$, where $\tau$ is the inverse-transpose automorphism of $L$.
\end{abstract}


\maketitle

\section{Introduction}

The {\em spectrum} $\omega(G)$ of a finite group $G$ is the set of orders of its elements, and groups with equal spectra are said to be {\em isospecrtral}.
This paper is a part of a larger investigation devoted to recognition of simple groups by spectrum. To solve the {\em problem of recognition by spectrum} for a  given
finite nonabelian simple group $S$ is to  describe (up to isomorphism) finite groups that are isospectral to $S$. As a  working hypothesis of that investigation, it was conjectured that a finite group $G$ isospectral to a ``sufficiently large'' simple group $S$ must be an almost simple group with socle $S$, i.\,e. $S\leq G\leq \Aut S$. In 2015, this conjecture was proved with the following precise meaning of the term '`sufficiently large'':
$S$ is a linear or unitary group of dimension larger than 44, or $S$ is a symplectic or orthogonal group of dimension larger than 60, or
$S$ is one of the sporadic, alternating and exceptional groups of Lie type other than  $J_2$, $A_6$, $A_{10}$, and $^3D_4(2)$ (see \cite{15VasGr1}).
Thus for a vast class of simple groups, the initial problem of recognition by spectrum was reduced to a more specific problem of describing (up to isomorphism) almost simple groups with socle $S$ that are isospectral to $S$,
and this is the problem we address in this paper.

For brevity, we refer to a finite almost simple group with socle $S$ that is isospectral to $S$ as {\em admissible} for $S$. Clearly we are interested in non-trivial admissible groups, i.\,e. other than $S$ itself.
It is not hard to check that there are no non-trivial admissible groups for the alternating groups. The information collected in \cite{85Atlas} allows to verify that the sporadic groups do not possess not-trivial admissible
groups either. One of the first examples of non-trivial admissible groups was discovered by Mazurov \cite{94Maz.t}: he showed that the finite groups isospectral to $PSL_3(5)$ are exactly $PSL_3(5)$ and its extension by the graph automorphism. Later Zavarnitsine \cite{04Zav} provided an example demonstrating that the number of admissible groups can be arbitrarily large: $PSL_3(7^{3^k})$ has exactly $k+1$ admissible groups, including itself,
and these are precisely extensions by field automorphisms.

Admissible groups are described for all exceptional groups of Lie type (see \cite{16Zve.t} for references),  $PSL_2(q)$ (see \cite{94BrShi}), $PSL_3(q)$ and $PSU_3(q)$ (see \cite{04Zav, 06Zav.t}), classical groups over fields of characteristic 2 \cite{08Gr.t, 13GrShi, 14Zve}, and symplectic and odd-dimensional orthogonal groups over fields of odd characteristic \cite{16Gr.t}. It is worth noting that for all of these groups, there is $\alpha\in \Aut S$ such that $S_0=\langle S, \alpha \rangle$ is admissible and any other admissible group is conjugate in $\Aut S$  to a subgroup of $S_0$; in other words, any admissible group is a cyclic extension of $S$ and
up to isomorphism there is a unique maximal admissible group. Below we will see that not all simple groups satisfy the latter property.

The main result of this paper is a description of admissible groups for linear and unitary groups over fields of odd characteristic (Theorems 1, 2 and 3).
Also we explicitly describe spectra of some almost simple groups with linear or unitary socle (Lemmas \ref{l:reduce}, \ref{l:graph_odd}, and \ref{l:graph_even}).
In the rest of this section, we introduce the notation used in the theorems and then state our results.

Throughout this paper, $p$ is a prime, $F$ is the algebraic closure of the field of order $p$, $H=GL_n(F)$, with matrices acting on row vectors by right multiplication, and $\tau$ is the inverse-transpose map $g\mapsto g^{-\top}$ of $H$. If $q$ is a power of $p$, then $F_q$ denotes the subfield of $F$ of order $q$, $\lambda_q$ denotes a fixed primitive element of $F_q$ and $\varphi_q$ denotes the standard Frobenius endomorphism of $H$ of level $q$, i.\,e. the endomorphism induced by raising matrix entries to the $q$th power. We identify $GL_n(q)$ with $C_H(\varphi_q)$ and $GU_n(q)$ with $C_H(\varphi_q\tau)$.

We write $GL_n^+(q)$ instead of $GL_n(q)$ and $GL_n^-(q)$ instead of $GU_n(q)$ and use a similar agreement for $PGL_n(q)$, $SL_n(q)$, and $PSL_n(q)$. For $\varepsilon\in\{+,-\}$,  we shorten $\varepsilon 1$ to $\varepsilon$ in arithmetic expressions.

As usual, by $(a_1,\dots,a_s)$ we denote the greatest common divisor of positive integers $a_1,\dots,a_s$, and by $[a_1,\dots,a_s]$ we denote their least common multiple. If $a$ and $b$ are positive integers, then
$\pi(a)$ denotes the set of prime divisors of $a$, $(a)_b$ denotes the largest divisor $c$ of $a$ such that $\pi(c)\subseteq \pi(b)$ and $(a)_{b'}$ denotes the number
$a/(a)_b$.

Let $L=PSL_n^\varepsilon(q)$, where $q=p^m$, and define $d=(n,q-\varepsilon)$. We write $\delta=\delta(\varepsilon q)$ to denote the diagonal automorphism of $L$ induced by $\diag(\lambda, 1,\dots,1)$, where $\lambda$ is a primitive $q-\varepsilon$th root of unity in $F_{q^2}$. We denote by $\varphi$ the field automorphism induced by $\varphi_p$. The automorphism induced by $\tau$ is denoted by the same letter. The image of $\alpha\in\Aut L$ in $\Out L$ is denoted by $\overline \alpha$.

When $n\geqslant 3$, the inverse-transpose automorphism is outer and $\Out L$ has the following presentation (we omit overbars for convenience):
\begin{gather*}\langle \delta, \varphi,\tau\mid \delta^d=\varphi^m=\tau^2=[\varphi,\tau]=1, \delta^{\varphi}=\delta^p, \delta^\tau=\delta^{-1}\rangle\quad\text{ if }\varepsilon=+,\\
\langle \delta, \varphi,\tau\mid \delta^d=\tau^2=[\varphi,\tau]=1, \varphi^m=\tau, \delta^{\varphi}=\delta^p, \delta^\tau=\delta^{-1}\rangle\quad\text{if }\varepsilon=-.\end{gather*}

Theorem \ref{t:graph} is concerned with the extension by the inverse-transpose automorphism $\tau$.
A criterion of admissibility of this extension is not very short, so it seems reasonable to write up it separately.

\begin{theorem}\label{t:graph}
Let $L=PSL_n^\varepsilon(q)$, where $n\geqslant 3$, $\varepsilon\in\{+,-\}$ and $q$ is odd, and $G=L\rtimes\langle \tau\rangle$. Then either $\omega(G)=\omega(L)$ or one of the following holds:

\begin{enumerate}
 \item  $n=p^{t-1}+2$ with $t\geqslant 1$, $q\equiv - \varepsilon\pmod{4}$ and $4p^t\in\omega(G)\setminus\omega(L)$;
 \item  $n=2^{t}+1$ with  $t\geqslant 1$, $(n,q-\varepsilon)>1$ and $2(q^{(n-1)/2}-\varepsilon^{(n-1)/2})\in\omega(G)\setminus\omega(L)$;
 \item $n=p^{t-1}+1$ with  $t\geqslant 1$ and $2p^t\in\omega(G)\setminus\omega(L)$;
 \item $n$ is even, $(n)_2\leqslant (q-\varepsilon)_2$, $q\equiv \varepsilon\pmod 4$ and $q^{n/2}+\varepsilon^{n/2}\in\omega(G)\setminus\omega(L)$;
 \item $n$ is even,  $(n)_{2'}>3$, $(n,q-\varepsilon)_{2'}>1$ and $2[q^{(n)_2}-1,q^{n/2-(n)_2}+\varepsilon^{n/2-(n)_2}]\in\omega(G)\setminus\omega(L)$.%
\end{enumerate}
\end{theorem}

Theorems \ref{t:main} and \ref{t:mainu} describe admissible groups appealing to Theorem \ref{t:graph}. As we mentioned,  the admissible groups for $PSL_3(q)$ and $PSU_3(q)$ were determined by Zavarnitsine \cite{04Zav, 06Zav.t},
so we do not consider these groups. However, we include them into the statements of the theorems for completeness. Observe that two almost simple groups with socle $S$ are isomorphic if and only if their images in $\Out S$ are conjugate.
Thus to describe admissible groups up to isomorphism, it is sufficient to describe their images in $\Out S$ up to conjugacy. We refer to $\alpha\in \Aut S$ as admissible if $\langle S,\alpha\rangle$ is admissible.

\begin{theorem}\label{t:main}
Let $L=PSL_n(q)$, where $n\geqslant 3$,  $q=p^m$ and $p$ is an odd prime.
Let $d=(n,q-1)$, $b=((q-1)/d,m)_d$, $\eta=\delta^{(d)_{2'}}$, $\phi=\varphi^{m/(b)_2}$, and
$$\psi=\begin{cases}\varphi^{m/(b)_{2'}}&\text{if $n\neq 4$ or $12\nmid q+1$};\\ \varphi^{(m)_{3'}}&\text{if $n=4$ or $12\mid q+1$}.\end{cases}$$
Suppose that  $L<G\leq \Aut L$. Then $\omega(G)=\omega(L)$ if and only if $G/L$ is conjugate in
$\Out L$ to a subgroup of $\langle \overline\alpha\rangle$, where $\alpha$ is one of the elements specified in Table $1$.
\end{theorem}

\begin{table}
\caption{}

\begin{tabular}{|l|l|c|cccc|}
\hline

\multicolumn{2}{|c|}{\multirow{2}*{Conditions on $L$}}&\multicolumn{5}{|c|}{$\alpha=\gamma\beta=\beta\gamma$}\\
\cline{3-7}
\multicolumn{2}{|c|}{}&$\gamma$&\multicolumn{4}{|c|}{$\beta$}\\
\hline
$n=p^{t}+1$, &$(b)_2>2$ or $(n)_2<(p-1)_2$& $1$&\multicolumn{4}{|c|}{$\varphi^{m/2}\tau\eta$}\\
\cline{2-7}
$n\neq 2^u+2$&$(m)_2=2$, $(n)_2<(p+1)_2$& $1$&\multicolumn{4}{|c|}{$\varphi^{m/2}\eta$}\\
\hline
$n\neq p^{t}+1$,&{$\tau$ not admissible, $|\psi|>1$}&$\psi$&\multicolumn{4}{|c|}{$1$}\\
\cline{2-7}
$b$ odd&{$\tau$ admissible}&$\psi$&\multicolumn{4}{|c|}{$\tau$}\\
\hline
\multirow{4}{2.4cm}{$n\neq p^{t}+1$, $b$~even, \\$p\equiv\epsilon\,(\text{mod}\,4)$}&$n=p^{s}+2^u+1$&$\psi$&$\phi$,& $\phi^i\tau$&&\\
\cline{2-7}
&$n\neq p^{s}+2^u+1$,$(n)_2\geqslant(p-\epsilon)_2$&$\psi$&$\phi$,& $\phi^i\tau$,&$\phi^{2j}\tau\eta$&\\
\cline{2-7}
&$n\neq p^{s}+2^u+1$, $(n)_2<(p-\epsilon)_2$, $\epsilon=+$&$\psi$&$\phi$, &$\phi^i\tau$,&$\phi^{2j}\tau\eta$,&$\phi\tau\eta$\\
\cline{2-7}
&$n\neq p^{s}+2^u+1$, $(n)_2<(p-\epsilon)_2$, $\epsilon=-$&$\psi$&$\phi$, &$\phi^i\tau$,&$\phi^{2j}\tau\eta$,&$\phi\eta$\\
\hline
\multicolumn{2}{|c|}{$t,u>0$, $s\geqslant 0$}&\multicolumn{2}{c}{}&\multicolumn{1}{c}{$2i\mid (b)_2$,}&$4j\mid (b)_2$&\\
\hline

\end{tabular}
\end{table}

As an example of applying Theorem \ref{t:main}, let us consider admissible groups for $PSL_4(25)$. For this group, $d=4$, $b=2$, $\phi=\varphi$, $n\neq p^t+1$ and $n=p^0+2^1+1$. Thus
up to isomorphism there are two non-trivial admissible groups, namely, the extensions by $\varphi$ and by $\varphi\tau$. Clearly they are both maximal.

\begin{theorem}\label{t:mainu}
Let $L=PSU_n(q)$, where $n\geqslant 3$,  $q=p^m$ and $p$ be an odd prime. Let $d=(n,q+1)$, $b=((q+1)/d,m)_d$ and $$\psi=\begin{cases}\varphi^{2m/(b)_{2'}}&\text{if $n\neq 4$ or $12\nmid q-1$};\\ \varphi^{2(m)_{3'}}&\text{if $n=4$ and $12\mid q-1$}.\end{cases}$$
Suppose that $L<G\leq \Aut L$. Then $\omega(G)=\omega(L)$ if and only if $n-1$ is not a power of $p$ and $G/L$
is conjugate in $\Out L$ to a subgroup of $\langle \overline\alpha\rangle$, where
\begin{enumerate}
\item $\alpha=\psi$ if $\tau$ is not admissible and $|\psi|>1$;
\item $\alpha=\psi\tau$ if $\tau$ is admissible, $(n)_2>2$ and $n\geqslant 16$;
\item $\alpha=\psi\varphi^{(m)_{2'}}$ if $\tau$ is admissible and either $(n)_2\leqslant 2$ or $n\leqslant 12$.
\end{enumerate}
\end{theorem}

Returning now to the initial recognition problem, we state the following consequence of \cite[Theorem 1]{15Vas} and the above results.

\begin{cor*}
Let $L=PSL_n^\varepsilon(q)$, where $n\geqslant 45$, $\varepsilon\in\{+,-\}$ and $q$ is odd. A finite group is isospectral to $L$ if and only if
it is isomorphic to an almost simple group $G$ with socle $L$ and $G/L=\langle\overline\alpha\rangle$, where $\alpha$ is an identity or is as specified in Theorems $2$ and $3$.
\end{cor*}

\section{Spectra of classical groups and related number-theoretical lemmas}

In this section, we collect necessary information on spectra of classical groups and related number-theoretical
lemmas. Our notation for the classical groups follows that of \cite{85Atlas}. Recall some well-known isomorphisms between classical groups (see, for example, \cite[Proposition 2.9.1]{90KlLie}). If $q$ is odd, then
\begin{equation*} PSp_4(q)\simeq \Omega_5(q), \quad \Omega_6^\varepsilon(q)\simeq SL_4^\varepsilon(q)/\{\pm 1\}, \quad P\Omega_6^\varepsilon(q)\simeq PSL_4^\varepsilon(q),\end{equation*}
\begin{equation} \label{e:small} \Omega_4^+(q)\simeq SL_2(q)\circ SL_2(q), \quad P\Omega_4^+(q)\simeq PSL_2(q)\times PSL_2(q), \quad \Omega_4^-(q)\simeq PSL_2(q^2),\end{equation}
\begin{equation*} Sp_2(q)\simeq SL_2(q), \quad \Omega_3(q)\simeq PSL_2(q).\end{equation*}

Given a prime $r$, we write $\omega_{r'}(G)$ to denote the set of orders of elements of $G$ that are coprime to  $r$. In particular, if $G$ is a group of Lie type over a field of characteristic $p$, then $\omega_{p'}(G)$ is the set of orders of semisimple elements of $G$. By $\omega_{\tilde r}(G)$ we denote the difference
$\omega(G)\setminus\omega_{r'}(G)$.

\begin{lemma}\label{l:spec_an}
Let $n\geqslant 2$, $q$ be a power of a prime $p$, $\varepsilon\in\{+,-\}$ and let $G$ be $PGL_n^\varepsilon(q)$ or $PSL_{n}^\varepsilon(q)$.
Let $d=1$ in the first case and $d=(n,q-1)$ in the second case. Then $\omega(G)$ consists of all divisors of the following numbers:
 \begin{enumerate}
 \item $(q^n-\varepsilon^n)/((q-\varepsilon)d)$;
 \item $[q^{n_1}-\varepsilon^{n_1}, q^{n_2}-\varepsilon^{n_2}]/(n/(n_1,n_2), d)$, where  $n_i>0$ and $n_1 + n_2= n$;
 \item $[q^{n_1}-\varepsilon^{n_1}, \dots, q^{n_s}-\varepsilon^{n_s}]$, where $s\geqslant2$, $n_i>0$ and $n_1 + \dots + n_s = n$;
 \item $p^t(q^{n_1}-\varepsilon^{n_1})/d$, where $t\geqslant 1$, $n_1>0$ and $p^{t-1}+1+ n_1= n$;
\item $p^t[q^{n_1}-\varepsilon^{n_1}, \dots, q^{n_s}-\varepsilon^{n_s}]$, where $t\geqslant 1$, $s\geqslant2$, $n_i>0$ and $p^{t-1}+1+ n_1 + \dots + n_s = n$;
\item $(n,q-1)p^t/d$ if $n=p^{t-1} + 1$ for some $t \geqslant 1$.
\end{enumerate}

\end{lemma}

\begin{proof}[{\bf Proof}] See \cite[Corollaries 2 and 3]{08But.t}.

\end{proof}

\begin{lemma}\label{l:spec_cn_bn}
Let $n\geqslant 1$, $q$ be a power of an odd prime $p$ and let $G$ be one of the groups $Sp_{2n}(q)$, $PSp_{2n}(q)$, and $\Omega_{2n+1}(q)$. Let $d=c=1$ if $G=Sp_{2n}(q)$;
$d=2$ and $c=1$ if $G=PSp_{2n}(q)$ or $G=\Omega_5(q), \Omega_3(q)$; and $d=c=2$ if $G=\Omega_{2n+1}(q)$ and $n\geqslant 3$. Then  $\omega(G)$ consists of all divisors of the following numbers:

\begin{enumerate}
 \item $(q^n\pm 1)/d$;
 \item $[q^{n_1}\pm 1, \dots, q^{n_s}\pm1]$, where $s\geqslant2$, $n_i>0$ and $n_1 + \dots + n_s = n$;
 \item $p^t(q^{n_1}\pm 1)/c$, where $t\geqslant 1$, $n_1>0$ and $p^{t-1}+1+ 2n_1 = 2n$;
\item $p^t[q^{n_1}\pm 1, \dots, q^{n_s}\pm 1]$, where $t\geqslant 1$, $s\geqslant2$, $n_i>0$ and $p^{t-1}+1+ 2n_1 + \dots + 2n_s = 2n$;
\item $2p^t/d$ if $2n=p^{t-1} + 1$ for some  $t \geqslant 1$.
\end{enumerate}
\end{lemma}

\begin{proof}[{\bf Proof}]  See \cite[Corollaries 1, 2 and 6]{10But.t} for $n\geqslant 2$ and (\ref{e:small}) together with Lemma \ref{l:spec_an} for $n=1$.

\end{proof}

\begin{lemma}\label{l:spec_dn}
Let $n\geqslant 2$,  $q$ be a power of an odd prime $p$ and $\varepsilon\in\{+,-\}$. Then   $\omega_{p'}(\Omega_{2n}^\varepsilon(q))$ consists of all divisors of the following numbers:
\begin{enumerate}
\item $(q^n-\varepsilon)/2$;
\item $[q^{n_1}-\kappa_1, \dots, q^{n_s}-\kappa_s]$, where $s\geqslant 2$, $\kappa_i\in\{+,-\}$, $n_i>0$, $n_1+\dots+n_s=n$ and $\kappa_1\kappa_2\dots\kappa_s=\varepsilon$;
\end{enumerate}
and $\omega_{p'}(P\Omega_{2n}^\varepsilon(q))$ consists of all divisors of the following numbers:

\begin{enumerate}
\item $(q^n-\varepsilon)/(4,q^n-\varepsilon)$;
\item $[q^{n_1}-\kappa, q^{n_2}-\varepsilon\kappa]/e$, where $\kappa\in\{+,-\}$, $n_1,n_2>0$, $n_1+n_2=n$;  $e=2$ when $(q^{n_1}-\kappa)_2=(q^{n_2}-\varepsilon\kappa)_2$ and $e=1$ otherwise;
\item $[q^{n_1}-\kappa_1, \dots, q^{n_s}-\kappa_s]$, where $s\geqslant 3$, $\kappa_i\in\{+,-\}$, $n_i>0$, $n_1+\dots+n_s=n$ and $\kappa_1\kappa_2\dots\kappa_s=\varepsilon$.
\end{enumerate}
\end{lemma}
\begin{proof}[{\bf Proof}] 
For $n\geqslant 4$, see \cite[Theorem 6]{07ButGr.t}. If $n=2,3$, the assertion follows from the isomorphisms given in  (\ref{e:small}) and Lemma \ref{l:spec_an}.
\end{proof}

Let $k\geqslant 3$ and $q$ be an integer whose absolute value is greater than one. 
A primitive prime divisor of $q^k-1$ is a prime $r$ such that $r$ divides $q^k-1$ and does not divide $q^i-1$ for  any $i<k$. The set of primitive prime divisors of $q^k-1$ is denoted by $R_k(q)$, and 
$r_k(q)$ denotes some fixed element of $R_k(q)$.

\begin{lemma}[ (Zsigmondy {\cite{Zs}})]\label{l:zsi}
Let $q$ be a prime power, $\varepsilon\in\{+,-\}$, and $k\geqslant 3$. Suppose that $(q,\varepsilon,k)\neq (2,+,6)$, $(2,-,3)$. Then $R_k(q)$ is not empty.
\end{lemma}

Two following results are well-known.

\begin{lemma}\label{l:gcd}
Let $q>1$, $k$ and $l$ be positive integers, and  $\varepsilon\in\{+,-\}$. Then
\begin{enumerate}
\item $(q^k-1,q^l-1)=q^{(k,l)}-1$;
\item $(q^k+1,q^l+1)$ is equal to $q^{(k,l)}+1$ if $(k)_{2}=(l)_{2}$ and to $(2,q+1)$ otherwise;
\item $(q^k-1,q^l+1)$ is equal to $q^{(k,l)}+1$ if $(k)_{2}>(l)_{2}$ and to $(2,q+1)$ otherwise;
\item $((q^k-\varepsilon^k)/(q-\varepsilon),k)=(q-\varepsilon,k)$;
\item if $(k,l)=1$, then $(q^l-\varepsilon^l)/(q-\varepsilon)$ divides $(q^{lk}-\varepsilon^{kl})/(q^k-\varepsilon^k)$  and $(q^l-\varepsilon^l)/(n,q-\varepsilon)$ divides $(q^{lk}-\varepsilon^{lk})/(n,q^k-\varepsilon^k)$ for any positive integer $n$.
\end{enumerate}
\end{lemma}

\begin{lemma}\label{l:r-part}
Let $q>1$ and $k$ be positive integers and $\varepsilon\in\{+,-\}$.
\begin{enumerate}
\item If an odd prime $r$ divides $q-\varepsilon$, then $(q^k-\varepsilon^k)_{r}=(k)_{r}(q-\varepsilon)_{r}$.
\item If an odd prime $r$ divides $q^k-\varepsilon^k$, then $r$ divides $q^{(k)_{r'}}-\varepsilon^{(k)_{r'}}$.
\item If $4$ divides $q-\varepsilon$ and $k$ is odd, then $(q^k-\varepsilon^k)_{2}=(k)_{2}(q-\varepsilon)_{2}$.
\end{enumerate}
\end{lemma}

\begin{lemma}\label{l:2adj} Let $q$ be odd and $n\geqslant 4$ be even.
Then $q^{n/2}+\varepsilon^{n/2}\in\omega(PSL^\varepsilon_n(q))$ if and only if $(n)_2>(q-\varepsilon)_2$.
\end{lemma}

\begin{proof}[{\bf Proof}]  Let $d=(n,q-\varepsilon)$.
Since $a=q^{n/2}+\varepsilon^{n/2}$ is divisible by a primitive divisor $r_n(\varepsilon q)$, it follows from  \ref{l:spec_an} that $a\in \omega(PSL_n^\varepsilon(q))$ if and only if 
$a$ divides  $$\frac{q^n-1}{(q-\varepsilon)d}=\frac{(q^{n/2}+\varepsilon^{n/2})(q^{n/2}-\varepsilon^{n/2})}{(q-\varepsilon)d};$$ that is, if and only if $d$ divides $(q^{n/2}-\varepsilon^{n/2})/(q-\varepsilon)$.

If $n/2$ is odd, then by Lemma \ref{l:r-part}, we have that $(q^{n/2}-\varepsilon^{n/2})/(q-\varepsilon)$ is odd and, therefore, it is not divisible by $d$, which is even.
If $n/2$ is even and $(n)_2\leqslant (q-\varepsilon)_2$, then $q\equiv \varepsilon\pmod 4$ and $(d)_2=(n)_2$, and hence  $(q^{n/2}-\varepsilon^{n/2})_2/(q-\varepsilon)_2=(n)_2/2<(d)_2$.
Also, if  $(n)_2>(q-\varepsilon)_2$, then $$((q^{n/2}-\varepsilon)/(q-\varepsilon),d)=((q^{n/2}-\varepsilon)/(q-\varepsilon),q-\varepsilon,n)=(n/2,q-\varepsilon,n)=d,$$
where the second equality holds by Lemma \ref{l:gcd}(iv).
\end{proof}

\section{Extensions by field and graph-field automorphisms}

In this section, we derive some formulas concerning orders of elements in extensions of $PSL_n^\varepsilon(q)$ by field or graph-field automorphisms. Following the lines of the proof of \cite[Proposition  13]{06Zav.t}, we will exploit a correspondence between $\sigma$\nobreakdash-\hspace{0pt}conjugacy classes of $C_K(\sigma^k)$ and conjugacy classes of $C_K(\sigma)$, where $K$ is a connected linear algebraic group and $\sigma$ is a Steiberg endomorphism of $K$, i.\,e. a surjective endomorphism with finitely many fixed points. Also we will use a slight modification of this correspondence inspired by \cite[Theorem 2.1]{08GowVin}.

We begin with necessary notation and the Lang--Steinberg theorem. If $K$ is a group and $\sigma$ is an endomorphism of $K$, then we write $K_\sigma$ to denote $C_K(\sigma)$.

\begin{lemma}[Lang--Steinberg]
Let $K$ be a connected linear algebraic group and $\sigma$ be a surjective endomorphism of $K$ such that $K_\sigma$ is finite. Then the map $x\mapsto x^{-\sigma}x$ from $K$ to $K$ is surjective.
\end{lemma}

Recall that $F$ is the algebraic closure of the field of order $p$ and $\varphi_q$ is the endomorphism of $GL_n(F)$
raising matrix entries to the $q$th power, where $q$ is a power of $p$. An endomorphism $\sigma$ of a linear algebraic group $K$ is said to be a Frobenius endomorphism if there are an identification of $K$ with a closed subgroup of  $GL_n(F)$ and a positive integer $k$ such that $\sigma^k$ is induced by $\varphi_q$.
Clearly, if $\sigma$ is a Frobenius endomorphism, then $\sigma$ is a Steinberg endomorphism.

\begin{lemma}\label{l:shin}
Let $K$ be a connected linear algebraic group over $F$, $\alpha$ be a Frobenius endomorphism of $K$ and
$\tau$ be an automorphism of $K$ of order $2$ that commutes with $\alpha$. Let $k$ be a positive integer, $l\in\{0,1\}$ and let
$\beta$ be the automorphism of $K_{\alpha^k\tau^l}$ induced by $\alpha$.
Given $g\in K_{\alpha^k\tau^l}$, choose  $z\in K$ such that $g=z^{-\alpha}z$ and define $\zeta(g)=z^{\tau^l}z^{-\alpha^k}$. Then 
$\zeta(g)\in K_\alpha$, $(\beta g)^k=z^{-1}\tau^l \zeta(g)z$ and the map $[\beta g]\mapsto [\tau^l\zeta(g)]$ is a one-to-one correspondence between the $K_{\alpha^k\tau^l}$-conjugacy classes in
the coset $\beta K_{\alpha^k\tau^l}$ and the $K_\alpha$-conjugacy classes in the coset $\tau^l K_\alpha$.
\end{lemma}

\begin{proof}[{\bf Proof}]  Observe that $\alpha^k$ and $\alpha^k\tau$ are Frobenius endomorphisms of $K$.

Let $g,z\in K$,  $g=z^{-\alpha}z$ and $h=z^{\tau^l}z^{-\alpha^k}$. Define $f=\tau^lh=z\tau^iz^{-\alpha^k}$. Then
\begin{equation}\label{e:equiv} g^{\alpha^k\tau^l}=g\Leftrightarrow  \tau^l z^{-\alpha^{k+1}}z^{\alpha^k}\tau^l=z^{-\alpha}z\Leftrightarrow z^{\alpha}\tau^l z^{-\alpha^{k+1}}=z\tau^l z^{-\alpha^k}\Leftrightarrow f^{\alpha}=f.\end{equation}
This if $g\in K_{\alpha^k\tau^l}$, then $f\in \tau^l K_\alpha$. Furthermore, $(\beta g)^k=\beta^kg^{\alpha^{k-1}}\dots g=\tau^lz^{-\alpha^k}z=z^{-1}fz$.

Let $g_1=z_1^{-\alpha}z_1\in K_{\alpha^k\tau^l}$ and $g_2=z_2^{-\alpha}z_2\in K_{\alpha^k\tau^l}$.
Suppose that $x^{-1}\alpha g_1x=\alpha g_2$ for some $x\in K_{\alpha^k\tau^l}$. Then writing $y=z_1xz_2^{-1}$, we have  $y^\alpha=y$ and
$$f_1=z_1(\beta g_1)^k z_1^{-1}=z_1x(\beta g_2)^k x^{-1}z_1^{-1}=yf_2y^{-1}.$$
Conversely, if $f_1=yf_2y^{-1}$ with $y\in K_\alpha$, then $x=z_1^{-1}yz_2$ satisfies $x^{\alpha^k\tau^l}=x$ and
$$x^{-1}\alpha g_1x=z_2^{-1}y^{-1}z_1\alpha z_1^{-\alpha}z_1z_1^{-1}yz_2=z_2^{-1}y^{-1}\alpha yz_2=z_2^{-1}\alpha z_2=\alpha g_2.$$
It follows that the map under consideration translate conjugacy classes to conjugacy classes and is injective. Furthermore, the conjugacy class $[f]$ 
does not depend on the choice of $z$ in the equality $g=z^{-\alpha}z$.

Now let $h\in K_\alpha$. By the Lang--Steinberg theorem, there is $z\in K$ such that $h^{-\tau^l}=z^{\alpha^k\tau^l}z^{-1}$, and therefore $\tau^lh=z\tau^lz^{-\alpha^k}$. Then (\ref{e:equiv})
implies that $g=z^{-\alpha}z$ lies in $K_{\alpha^k\tau^l}$, with $[\beta g]$ mapping to $[\tau^lh]$. Thus the map is surjective, and the proof is complete.

\end{proof}

It should be noted that some special cases of Lemma \ref{l:shin} were proved in \cite[Lemma 2.10]{12FulGur} ($l=0$) and \cite[Theorem 2.1]{08GowVin} ($l=1$ and $k=1$).

We use Lemma \ref{l:shin} to establish the following result generalizing \cite[Corollary 14]{06Zav.t}.

\begin{lemma} \label{l:reduce} Let $L=PSL_n(q)$ and $U=PSU_n(q)$, where $n\geqslant 3$ and $q=p^m$. Let $k$ divides $m$,  $q=q_0^k$ and $\beta=\varphi^{m/k}$. Then for any $i$, we have the following:
\begin{enumerate}
\item $\omega(\beta\delta^i(q)L)=k\cdot \omega(\delta^i(q_0)PSL_n(q_0));$
\item if $k$ is even, then  $\omega(\beta\tau\delta^i(q)L)=k\cdot \omega(\delta^i(-q_0)PSU_n(q_0));$
\item if $k$ is odd, then $\omega(\beta\tau\delta^i(-q)U)=k\cdot\omega(\delta^i(-q_0)PSU_n(q_0));$
\item if $k$ is odd, then $\omega(\beta\tau\delta^i(q)L)=k\cdot\omega(\tau\delta^i(q_0)PSL_n(q_0));$
\item $\omega(\beta\delta^i(-q)U)=k\cdot \omega(\tau\delta^i(q_0)PSL_n(q_0)).$
\end{enumerate}
In particular, $\omega(\tau\delta^i(q)L)=\omega(\tau\delta^i(-q)U).$
\end{lemma}

\begin{proof}[{\bf Proof}]  Denote the center of  $H=GL_n(F)$ by $Z$. In particular, we write $|gZ|$ to denote the projective order of a matrix $g\in H$.

We apply Lemma \ref{l:shin} to $H$ with $\alpha=\varphi_p^{m/k}$ and $l=0$. Observe that $H_{\alpha^k}=GL_n(q)$ and $H_{\alpha}=GL_n(q_0)$. Let $g\in GL_n(q)$ and $h=\zeta(g)$ be the element of $GL_n(q_0)$
defined in Lemma \ref{l:shin}.
Since $(\beta g)^k$ is conjugate to $h$ in $H$, we see that 
$$|\beta gZ|=k\cdot |(\beta gZ)^k|=k\cdot |hZ|,$$ and also 
$$\det h=\det(g^{\beta^{k-1}}\dots g)=(\det g)^{q_0^{k-1}+\dots+1}=(\det g)^{(q-1)/(q_0-1)}.$$
Let $\lambda=\lambda_{q}$ and $\lambda_0=\lambda_{q_0}$. The condition $gZ\in\delta^iL$ is equivalent to $\det g\in\langle \lambda^i,\lambda^n\rangle=\langle \lambda^{(i,n)}\rangle$.
Since $\lambda^{(q-1)/(q_0-1)}$ is a primitive element of $F_{q_0}$, it follows that $\det g\in \langle \lambda^{(i,n)}\rangle$ if and only if
$\det h \in \langle \lambda_0^{(i,n)}\rangle$. Thus (i) holds. Similarly, (ii) and (iii) follow from Lemma \ref{l:shin}  with $\alpha=\varphi_p^{m/k}\tau$, $l=0$.

Let $k$ be odd and $\alpha=\varphi_p^{m/k}\tau$. Then $H_{\alpha^k\tau}=GL_n(q)$ and $H_{\alpha}=GU_n(q_0)$. Applying Lemma \ref{l:shin} with $l=1$, we construct from an element $g\in GL_n(q)$
an element $h\in GU_n(q_0)$ such that $g=z^{-\alpha}z$ and $h=z^{\tau}z^{-\alpha^k}$ for some $z\in H$.
As in the previous case, we deduce that $|\beta\tau gZ|=k\cdot |\tau hZ|$. Furthermore, 
$\det g=(\det z)^{q_0+1}$ and $\det h=(\det z)^{q-1}$. Since $(\det g)^{q-1}=1$, it follows that $(\det z)^{(q_0+1)(q-1)}=1$. Let $\lambda=\lambda_q$, $\lambda_0=\lambda_{q_0}^{q_0-1}$ and
$\mu$ be a primitive $(q_0+1)(q-1)$th root of unity in $F$. Then
$\det g\in \langle \lambda^{(i,n)}\rangle$ if and only if $\det z\in \langle \mu^{(i,n)}\rangle$, which is equivalent to $\det h\in \langle \lambda_0^{(i,n)}\rangle$.
Hence 
 \begin{equation}\label{e:3lu}\omega(\beta\tau\delta^i(q)L)=k\cdot\omega(\tau\delta^i(-q_0)PSU_n(q_0)).\end{equation}

Similarly, taking $\alpha=\varphi_p^{m/k}$ and $l=1$, we prove (v).
Applying (v) with $k=1$ and observing that in this case $\beta$ acts on $U$ in the same way as $\tau$, we have $$\omega(\tau\delta^i(q)L)=\omega(\tau\delta^i(-q)U).$$
Now this equality and (\ref{e:3lu}) imply (iv).

\end{proof}

Note that for unitary groups, (iii) and (v) of Lemma \ref{l:reduce} cover all possibilities for an element $\gamma\in \langle \varphi \rangle$: if $|\gamma|=2k$, then $\langle\gamma \rangle=\langle \varphi^{m/k}\rangle$; 
while if $|\gamma|=k$ with $k$ odd, then $\langle \gamma\rangle=\langle \varphi^{m/k}\tau\rangle$ because $(m+m/k, 2m)=2m/k$.
Thus Lemma \ref{l:reduce} expresses the spectrum of an extension by a field or graph-field automorphism in terms of known spectra and the spectrum of the extension by graph automorphism, which we consider 
in the next section.

\section{Extension by graph automorphism}

This section is largely concerned with matrices, so we need to define some of them. We denote by $E$ the identity matrix whose size is clear from the context, and by $J_k$ the $k\times k$ unipotent Jordan 
block.

Recall that $\langle \overline\delta\rangle\rtimes \langle \overline\tau\rangle$ is a dihedral group of order $2(n,q-\varepsilon)$. It follows that for odd $n$, every $\tau\delta^i$ is conjugate to $\tau$ modulo $L$,
and hence \begin{equation}\label{e:nodd}\omega(\tau L)=\omega(\cup_{i=1}^{d}\tau\delta^iL)=\omega(\tau PGL_n(q)).\end{equation} If $n$ is even, then for any $i$, we have 
\begin{equation}\label{e:neven} \omega(\tau L)=\omega(\tau\delta^{2i}L)\text{\quad and \quad} \omega(\tau\delta L)=\omega(\tau\delta^{2i-1}L).\end{equation}

As we saw in Lemma \ref{l:reduce}, the cosets $\tau\delta^i(q)PSL_n(q)$ and $\tau\delta^i(-q)PSU_n(q)$ have the same orders of elements, so it suffices to describe $\omega(\tau\delta^i(q)PSL_n(q))$.

Since $|g\tau|=2|(g\tau)^2|=2|gg^\tau|$, it follows that the elements of $\omega(\tau PGL_n(q))$ are exactly twice the projective orders of elements of $$\Gamma_n(q)=\{gg^\tau\mid g\in GL_n(q)\}.$$
If $n$ is even, then $\omega(\tau L)$ and $\omega(\tau\delta L)$ are analogously related to the projective orders of elements of 
$$\Gamma_n^\square(q)=\{gg^\tau\mid g\in GL_n(q), \det g\in (F_q^\times)^2\}$$ and
$$\Gamma_n^\boxtimes(q)=\{gg^\tau\mid g\in GL_n(q), \det g\not\in (F_q^\times)^2\},$$
respectively.

A comprehensive treatment of the equation $h=gg^\tau$ for a given matrix $h$ is provided by Fulman and Guralnick in \cite{04FulGur}, and we use the terminology 
and some results of this paper. First of all, it is helpful to note that $h = gg^\tau$ yields \begin{equation}\label{e:class} xhx^{-1}=(xgx^\top)(xgx^\top)^\tau\end{equation} for any $x\in GL_n(q)$,
and hence  $h$ lies in $\Gamma_n(q)$, $\Gamma_n^\square(q)$ or $\Gamma_n^\boxtimes(q)$ if and only if the whole conjugacy class $[h]$ lies in the corresponding set. Thus we
may work not with individual matrices but with conjugacy classes. Recall that the conjugacy classes of $GL_n(q)$ are parametrized by collections of partitions 
$$\{\lambda_\phi\mid \phi\text{ is a monic irreducible polynomial over } F_q\}$$ such that $|\lambda_z|=0$ and $\sum_\phi\deg(\phi)|\lambda_\phi|=n$.
In this parametrization, the collection of partitions $\{\lambda_\phi\}$ corresponds to the class $[h]$ such that the multiplicity of $\phi^k$ as an elementary divisor of $h$ 
is equal to the multiplicity of parts of size $k$ in $\lambda_\phi$. We denote the partition 
corresponding under this parametrization to a class $[h]$ and a polynomial $\phi$ by $\lambda_\phi(h)$.

A criterion for a matrix $h$ to lie in $\Gamma_n(q)$ was obtained by Wall in \cite{63Wall}. In the same paper, Wall described the conjugacy classes of finite symplectic and orthogonal groups over fields of odd characteristic,  and it turns out that the matrices of $\Gamma_n(q)$ are very similar to symplectic and orthogonal ones. 

\begin{lemma}\label{l:resolve}
If $q$ is odd, then  $h\in \Gamma_n(q)$ if and only if $h$ satisfies the following:

\begin{enumerate}
\item $h$ is conjugate to $h^{-1}$;

\item all even parts of $\lambda_{z-1}(h)$ have even multiplicity;

\item all odd parts of $\lambda_{z+1}(h)$ have even multiplicity.

\end{enumerate}

\end{lemma}

\begin{proof}[{\bf Proof}] 
See \cite[Theorem 2.3.1]{63Wall}.
\end{proof}

\begin{lemma}\label{l:sym}
Let $h\in GL_n(q)$ and $q$ be odd.  Then $h$ is conjugate to a element of $Sp_n(q)$ if and only if $h$ satisfies the following:

\begin{enumerate}
\item $h$ is conjugate to $h^{-1}$;

\item all odd parts of $\lambda_{z-1}(h)$ and $\lambda_{z+1}(h)$ have even multiplicity.

\end{enumerate}

\end{lemma}

\begin{proof}[{\bf Proof}] 
See \cite[p. 36, case (B)($i$)]{63Wall}.
\end{proof}

\begin{lemma}\label{l:orth}
Let $h\in GL_n(q)$ and $q$ be odd. Then $h$ is conjugate to an element of $GO^\varepsilon_n(q)$ for some $\varepsilon$ if and only if $h$ satisfies the following:

\begin{enumerate}
\item $h$ is conjugate to $h^{-1}$;

\item all even parts of $\lambda_{z-1}(h)$ and $\lambda_{z+1}(h)$ have even multiplicity.

\end{enumerate}

Suppose that $n$ is even and $h$ is a unipotent element satisfying $(ii)$. If $\lambda_{z-1}(h)$ has odd parts, then 
$h$ is conjugate to an element of $GO_n^+(q)$ and also to an element of $GO_n^-(q)$. If there are no odd parts, then $h$ is not conjugate to an 
element of $GO_n^-(q)$.

Suppose that $n$ is even, $|\lambda_{z-1}(h)|=|\lambda_{z+1}(h)|=0$ and $h$ satisfies $(i)$. Then $h$ is conjugate to an element 
of only one of the groups $GO_n^+(q)$ and $GO_n^-(q)$.

\end{lemma}

\begin{proof}
See \cite[p. 38, Case (C)($i$,$i'$)]{63Wall}.
\end{proof}

Next we establish a necessary and sufficient condition for $h$ to lie in $\Gamma_n^\square(q)$ or $\Gamma_n^\boxtimes(q)$. Let $h\in GL_n(q)$, $f$ be the characteristic polynomial of $h$, and suppose that 
$f=(z-1)^{n_1}(z+1)^{n_2}f_0$ with $(f_0,z^2-1)=0$. Then $h$ is conjugate to a block diagonal matrix 
\begin{equation}\label{e:normal} \diag(h_1,h_{-1},h_0)\end{equation} with blocks of dimension $n_1$, $n_2$ and $n-n_1-n_2$, respectively, corresponding to this factorization of $f$.
We refer to the matrix in (\ref{e:normal}) as a normal form of $h$.

Let $h\in \Gamma_n(q)$. We may replace $h$ by its normal form $\diag(h_1,h_{-1},h_0)$ with blocks of dimensions $n_1$, $n_2$ and $n-n_1-n_2$, respectively, and denote by $V_1$ the subspace of $F_q^n$ spanned by the first  $n_1$ rows. Suppose that $h=gg^\tau$. Then $h^g=g^\tau g=h^\tau$. It follows that $V_1g$ is $h^\tau$-invariant and the characteristic polynomial of $h^\tau$  on ${V_1g}$ is equal to $(z-1)^{n_1}$, and hence
$V_1g=V_1$. The same is true for other blocks, and so $g$ is also block diagonal with blocks $g_1$, $g_{-1}$ and $g_0$ of the same dimensions as $h_1$, $h_{-1}$ and $h_0$, respectively (see also \cite[Lemma 8.2]{04FulGur}). Since  $h_i=g_ig_i^\tau$ for $i=1,-1,0$ and $\det g=\det g_1\det g_{-1}\det g_0$, it suffices to consider the following special cases: $f=(z-1)^n$, $f=(z+1)^n$, and $(f,z^2-1)=1$.

Recall that for odd $q$, the group $\Omega_{2n}^\varepsilon(q)$ is the kernel of the spinor norm $\theta: SO_{2n}^\varepsilon(q)\rightarrow F_q^\times/(F_q^\times)^2$.
The definition of the spinor norm in \cite[pp. 163--165]{09Tay} implies the following way to calculate it (see also \cite[Proposition 1.6.11]{13BHRD}).

\begin{lemma}\label{l:spinor}
Let $q$ be odd, $n$ be even,  $h\in SO_n^\varepsilon(q)$ and $B$ be the matrix of the invariant symmetric bilinear form of $SO_n^\varepsilon(q)$.
Suppose that $\det(E-h)\neq 0$. Then $\theta(h)\equiv \det((E-h)B)\pmod{(F_q^\times)^2}$.
\end{lemma}

\begin{lemma}\label{l:sg}
Let $q$ be odd, $n$ be even, $h\in \Gamma_n(q)$ and $f$ be the characteristic polynomial of $h$.
\begin{enumerate}
\item If $f=(z+1)^n$, then $h\not\in \Gamma_n^\boxtimes(q)$.
\item Let $f=(z-1)^n$. If $\lambda_{z-1}(h)$ has no odd parts, then $h\not\in \Gamma_n^\boxtimes(q)$; otherwise, $h\in \Gamma_n^\square(q)\cap\Gamma_n^\boxtimes(q)$.
\item Let $(f, z^2-1)=1$. Then $h$ is conjugate to an element of $SO_n^\varepsilon(q)$ for some unambiguously defined $\varepsilon$, and $h\in \Gamma^\square_n(q)$ if and only if 
$-h$ is conjugate to an element of $\Omega_n^\varepsilon(q)$.
\end{enumerate}

\end{lemma}

\begin{proof}[{\bf Proof}]  Let $h=gg^\tau$. Possibilities for $g$ in the first two cases are found in \cite[Section 8]{04FulGur}, and we use this result for calculations. 
For brevity, we write $x \equiv y$ to denote that $x \equiv y\pmod{(F_q^\times)^2}$.

(i) The group $GL_n(q)\rtimes \langle\tau\rangle$ can be embedded into $GL_{2n}(q)$, and we consider the Jordan decomposition of $g\tau$ in the letter group. 
So $g\tau = ug_1\tau= g_1\tau u$, where $u$ is unipotent and $(g_1\tau)^2 =-E$; i.\,e. $g_1$ is a skew-symmetric matrix.
Since the determinant of a skew-symmetric matrix is a square and $\det u=1$, we see that $\det g$ is also a square.

(ii) Similarly, $g\tau = ug_1\tau= g_1\tau u$, where $u$ is unipotent and $(g_1\tau)^2 = E$; i.\,e.  $g_1$ is a symmetric matrix.
Furthemore, $h=u^2$ and the equality $ug_1\tau= g_1\tau u$ is equivalent to the condition that $u$ preserves the bilinear form defining by $g_1$. By Lemma  \ref{l:orth},
if $\lambda_{z-1}(u)=\lambda_{z-1}(h)$ has odd parts, then $u$ is conjugate both to an element of $SO_n^+(q)$ and to an element of $SO_n^-(q)$, and so 
we can choose $g_1$ with any determinant, square or non-square. If $\lambda(u)$ has no odd parts, then $u$ is conjugate to an element of $SO_n^+(q)$ only, and since 
the multiplicities of even parts are even, $n$ is divisible by 4. In this case $\det g_1\equiv \det E$.

(iii) The fact that $h$ is conjugate to an element of $SO_n^\varepsilon(q)$, where $\varepsilon$ is unambiguously defined, follows from Lemmas \ref{l:resolve} and \ref{l:orth}. 
By Lemma \ref{l:spinor}, we have $\theta(h)\equiv \det(E-h)\det B$, where $B$ is the matrix of the symmetric bilinear form preserved by $h$.
Observing that $\theta(-E)\equiv\det B$, we have $$\theta(-h)\equiv\det(E-h)=\det(E-gg^\tau)=\det(g^\top-g)\det g^\tau\equiv \det g,$$
where the final equivalence holds because the determinant of the skew-symmetric matrix $g^\top-g$ is a square. The proof is complete.

\end{proof}

We are ready to find $\omega(\tau PGL_n(q))$ and $\omega(\tau PSL_n(q))$.

\begin{lemma}\label{l:graph_odd} Let $q$ and $n\geqslant 3$ be odd. Then $$\omega(\tau PSL_n(q))=\omega(\tau PGL_n(q))=2\cdot \omega(Sp_{n-1}(q)).$$
\end{lemma}

\begin{proof}[{\bf Proof}]  The first equality was established in (\ref{e:nodd}). To prove the second one, we need to show that the set of the projective orders of 
elements of  $\Gamma_n(q)$ is equal to $\omega(Sp_{n-1}(q))$.

Let $h\in \Gamma_n(q)$. Since $n$ is odd, Lemma \ref{l:resolve} implies that 1 is an eigenvalue of $h$. It follows that the projective order of $h$ is equal to the ordinary order. 

We show first that $|h|\in\omega(Sp_{n-1}(q))$. Let $\diag(h_1,h_{-1},h_0)$ be a normal form of $h$ as in (\ref{e:normal}) and denote the dimension of $h_1$ by $k$. Then $k$ is odd.
By Lemmas \ref{l:resolve} and \ref{l:sym}, the unipotent matrix $h_1$ is conjugate to an element of $SO_k(q)$ and the matrix $\diag(h_{-1},h_0)$ is conjugate to an element of $Sp_{n-k}(q)$.
If $k=1$, then there is nothing to prove. If $k>1$, then $|h_1|\in\omega_p(SO_k(q))=\omega_p(Sp_{k-1}(q))$, and hence $|h|\in\omega(Sp_k(q)\times Sp_{n-k}(q))\subseteq\omega(Sp_{n-1}(q))$.

It follows from Lemmas \ref{l:resolve} and \ref{l:sym} that a semisimple matrix lies in $\Gamma_n(q)$ if and only if it is conjugate to a matrix of the form $\diag(1,h')$ with $h'\in Sp_{n-1}(q)$, and thus $\omega_{p'}(\Gamma_n(q))=\omega_{p'}(Sp_{n-1}(q))$.

Let $a\in \omega(Sp_{n-1}(q))$ and $|a|_p=p^t>1$. If $a=2p^t$ and $n-1=p^{t-1}+1$ (that is, the condition from (v) of Lemma \ref{l:spec_cn_bn} holds), then we define $h=\diag(1, -J_{n-1})$. If $a=p^t$ or $n-1>p^{t-1}+1$, then 
there is a semisimple matrix $h_s\in Sp_{l}(q)$, where $n-1=p^{t-1}+1+l$, such that $a=p^t|h_s|$, and we define $h=\diag(J_{p^{t-1}+2}, h_s)$. It is easy to see that $h\in \Gamma_n(q)$ and $|h|=a$. The proof is complete.
\end{proof}

\begin{lemma}\label{l:graph_even}
Let $q$ be odd and $n\geqslant 4$ be even. Then $$\omega(\tau PGL_n(q))=2\cdot\omega(PSp_n(q)),$$ $$\omega_{p'}(\tau PSL_n(q))=2\cdot\omega_{p'}(P\Omega_n^+(q))\cup 2\cdot\omega_{p'}(P\Omega_n^-(q)).$$
If $n>4$, then $$\omega_{\tilde p}(\tau PSL_n(q))=2\cdot\omega_{\tilde p}(\Omega_{n+1}(q)).$$
The set $\omega_{\tilde p}(\tau PSL_4(q))$ consists of all multiples of $p$ dividing $p(q\pm 1)$ if $p>3$, and it consists of these multiples together with $18$ if $p=3$.
\end{lemma}

\begin{proof}[{\bf Proof}]   It suffices to find the sets of the projective orders of elements of $\Gamma_n(q)$ and $\Gamma_n^\square(q)$. For brevity, we denote these sets by 
$\omega(P\Gamma_n(q))$ and $\omega(P\Gamma_n^\square(q))$ respectively. Also we denote by $Z$ the center of $GL_n(q)$.

Let $h\in \Gamma_n(q)$ and $\diag(h_1,h_{-1},h_0)$ be a normal form of $h$. Since $n$ is even, it follows from Lemma \ref{l:resolve} that the dimension of $h_1$ is even too. We denote this dimension by $2k$ and  define $h_2=\diag(h_{-1},h_0)$. Then $h_1$ is conjugate to an element of $SO_{2k}^+(q)$ and $h_2$ is conjugate to an element of $Sp_{n-2k}(q)$.

Let $k>0$. Then $|h_1|\in\omega_p(SO_{2k}^+(q))=\omega_p(Sp_{2k-2}(q))$, and hence $$|h|\in \omega(Sp_{2k-2}(q)\times Sp_{n-2k}(q))\subseteq\omega(Sp_{n-2}(q)).$$
Using Lemmas \ref{l:spec_cn_bn} and \ref{l:spec_dn}, it is easy to check that $$\omega_{p'}(Sp_{n-2}(q))\subseteq \omega_{p'}(P\Omega_n^+(q))\cup \omega_{p'}(P\Omega_n^-(q))\subseteq \omega (PSp_n(q)),$$
$$\omega_{\tilde p}(Sp_{n-2}(q))\subseteq \omega_{\tilde p}(\Omega_{n+1}(q))\subseteq \omega_{\tilde p} (PSp_n(q)).$$
Thus the projective order of $h$ lies in the required sets.

Let $k=0$. Then $h=h_2$ and $|hZ|\in \omega(PSp_n(q))$. It follows that  $$(P\Gamma_n(q))\subseteq \omega(PSp_n(q)).$$

Suppose, in addition, that $h\in \Gamma_n^\square(q)$. Denote the dimension of $h_{-1}$ by $2l$. By Lemma \ref{l:sg}, we deduce that $h_{-1}\not\in \Gamma_{2l}^\boxtimes(q)$, and so
$-h_0$ is conjugate to an element of $\Omega_{n-2l}^\varepsilon(q)$ for some $\varepsilon$. Observe that $|hZ|=|-hZ|$.

Assume that $|-h_{-1}|>1$. Then $|-hZ|=|-h|=|-h_{-1}|\cdot|-h_0|$. Since $|-h_{-1}|\in\omega_p(Sp_{2l}(q))=\omega_p(\Omega_{2l+1}(q))$, we have $$|hZ|\in\omega_{\tilde p}(\Omega_{2l+1}(q)\times \Omega_{n-2l}^\varepsilon(q))\subseteq \omega_{\tilde p}(\Omega_{n+1}(q)).$$
Furthermore, if $n=4$ then either $|hZ|=9$ and $p=3$, or $|hZ|$ lies in $p\cdot\Omega_{2}^\varepsilon(q)$ and, therefore, divides $p(q-\varepsilon)/2$.

If $|-h_{-1}|=1$, then either $l=0$ and $h=h_0$, or $l>0$ and $|-hZ|=|-h_0|$. In either case, $|-hZ|\in\omega(P\Omega_n^\varepsilon(q))\subseteq\omega(\Omega_{n+1}(q))$. In particular, if $n=4$ and the order $|-hZ|$ is a multiple of 
$p$, then it divides $p(q\pm 1)/2$.

Now we prove the reverse containments. If  $h$ is semisimple, then $h\in \Gamma_n(q)$ if and only if $h$ is conjugate to an element of $Sp_n(q)$,
and hence $$\omega_{p'}(PSp_n(q))\subseteq \omega(P\Gamma_n(q)).$$

Let $a\in\omega(PSp_n(q))$ and $|a|_p=p^t>1$. Lemma \ref{l:spec_cn_bn} implies that there is a semisimple matrix $h_s\in Sp_{l}(q)$, where $n=p^{t-1}+1+l$, such that $a=p^t|h_s|$, and we define $h=\diag(-J_{p^{t-1}+1}, -h_s)$. It is easy to see that $h\in \Gamma_n(q)$ and $|hZ|=|-hZ|=|-h|=a$. Thus $$\omega_{\tilde p}(PSp_n(q)\subseteq \omega(P\Gamma_n(q)),$$ and the assertion concerning $\tau PGL_n(q)$ follows.

Let $h\in\Omega_n^\varepsilon(q)$ be semisimple and $\diag(h_1,h_{-1},h_0)$ be a normal form of $h$, with $h_1$, $h_{-1}$ of dimension $k$ and $l$ respectively. If $k>0$, then $h\in \Gamma_n^\square(q)$ by Lemma \ref{l:sg}(ii).
If $k=0$, then applying (i) or (iii) of Lemma \ref{l:sg} according as $l>0$ or $l=0$,  we deduce that $-h\in\Gamma_n^\square(q)$. In any case,  $|hZ|\in\omega(P\Gamma^\square_n(q))$.

Suppose that $n>4$ and $a\in\omega_{\tilde p}(\Omega_{n+1}(q))$, or $n=4$ and $a=pc$ with $c$ dividing $(q\pm 1)/2$, or $n=4$, $p=3$ and $a=9$. Let $|a|_p=p^t$. By Lemmas \ref{l:spec_cn_bn} and \ref{l:spec_dn}, there are  $\varepsilon\in\{+,-\}$ and a semisimple matrix $h_s\in \Omega^\varepsilon_{l}(q)$, where $n=p^{t-1}+1+l$, such that $a=p^t|h_s|$.
Defining $h=\diag(-J_{p^{t-1}+1}, -h_s)$, we see that $h\in\Gamma_n^\square(q)$ and $|hZ|=|-hZ|=|-h|=a$. The proof is complete.

\end{proof}

In contrast to the sets of the projective orders of elements of $\Gamma_n(q)$ and $\Gamma_n^\square(q)$, the corresponding set for $\Gamma_n^\boxtimes(q)$, where $n$ is even, is not in general 
closed under taking divisors. So we do not give a explicit description of the set $\omega(\tau\delta PSL_n(q))$ for even $n$. However, we derive some properties of this set.

\begin{lemma}\label{l:graph_diag} Let $q$ be odd and $n\geqslant 4$ be even. Then $\omega(\tau\delta PSL_n(q))\not\subseteq\omega(PSL_n(q))$.
If $k=2^s$, where $s\geqslant 0$, and $q=q_0^k$, then $k\cdot \omega(\tau\delta(q_0)PSL_n(q_0))=\not \subseteq \omega (PSU_n(q))$. In particular,
$\omega(\tau\delta(-q) PSU_n(q))\not\subseteq\omega(PSL_n(q))$.

\end{lemma}
\begin{proof}[{\bf Proof}]  Recall that $\omega(\tau\delta PSL_n(q))$ consists of the projective orders of matrices of $\Gamma_n^\boxtimes(q)$ multiplied by $2$.

There is an element $h\in SO_n^\epsilon(q)\setminus\Omega_n^\epsilon(q)$ whose projective order is equal to $(q^{n/2}-\epsilon)/2$. By Lemma \ref{l:sg}, it follows that  $-h\in \Gamma_n^\boxtimes(q)$, and so 
 $q^{n/2}\pm 1\in \omega (\tau\delta PSL_n(q))$.

Assume that $n>4$ and $l=(n-2)/2$ is odd. Then Lemma \ref{l:spec_cn_bn} together with the existence of primitive divisors $r_{l}(\pm q)$ implies that $p(q^{l}\pm 1)\in \omega(PSp_n(q))\setminus\omega(\Omega_{n+1}(q))$.
Applying Lemma \ref{l:graph_even}, we see that \begin{equation}\label{e:l} 2p(q^{l}\pm 1)\in \omega(\tau PGL_n(q))\setminus\omega(\tau PSL_n(q))\subseteq  \omega(\tau\delta PSL_n(q)).\end{equation}

Assume now that $n=4$ and $q\equiv \epsilon \pmod 4$. The numbers $p(q+1)$ and $p(q-\epsilon)$ lie in $\omega(PSp_4(q))$ and do not divide $p(q\pm 1)/2$ (if $q-1$ divides $(q+1)/2$, then $q=3$),
and hence \begin{equation}\label{e:4}2p(q+1), 2p(q-\epsilon)\in\omega(\tau\delta PSL_4(q)).\end{equation}

We prove the first assertion of the lemma and the second assertion for $k=1$ together. Suppose that $\omega(\tau\delta PSL_n(q))\subseteq\omega(PSL_n^\varepsilon(q))$ and as usual define $d=(n,q-\varepsilon)$.
Then $q^{n/2}+\varepsilon^{n/2}\in \omega(PSL_n^\varepsilon(q))$. By Lemma \ref{l:2adj}, this is equivalent to $(n)_2>(q-\varepsilon)_2$.
It follows that $l=(n-2)/2$ is odd and $(d)_2=(q-\varepsilon)_2$, and so $(q^{l}-\varepsilon)/d$ is odd. Then $2(q^{l}+\varepsilon)$ does not divide $(q^{n-2}-1)/d$ and, therefore, 
$2p(q^{l}+\varepsilon)\not\in\omega(PSL_n^\varepsilon(q))$. If $n>4$, this contradicts (\ref{e:l}). If $n=4$, then $q\equiv -\varepsilon\pmod 4$ and (\ref{e:4}) implies that
$2p(q+\varepsilon)\in \omega(\tau\delta PSL_4(q))$ yielding a contradiction.

Suppose now that $k>1$ and $k\cdot \omega(\tau\delta(q_0)PSL_n(q_0))\subseteq \omega (PSU_n(q))$. Observe that $(q+1)_2=2$and choose $\epsilon$ so that $q_0\equiv \epsilon \pmod 4$.
Let $n/2$ be odd. By assumption, $a=k(q_0^{n/2}-\epsilon)\in\omega(PSU_n(q))$. Since $a$ is a multiple of $r_{n/2}(\epsilon q_0)$ and $r_{n/2}(\epsilon q_0)\in R_{n/2}(q)=R_n(-q)$,
it follows that $a$ divides $c=(q^n-1)/((q+1)(n,q+1))$. This is a contradiction because  $$(a)_2=(q^{n/2}-1)_2>(c)_2=(q^n-1)_2/4.$$ Thus $l=(n-2)/2$ is odd and by (\ref{e:l}) and (\ref{e:4}), we have  $2p(q_0^{l}-\epsilon)\in \omega(\tau\delta PSL_n(q_0))$.
Then $a=2kp(q_0^{l}-\epsilon)\in \omega(PSU_n(q))$, and hence $a$ divides $p(q^{n-2}-1)/(n,q+1)$. This contradicts $(a)_2=(q^{n-2}-1)_2$.

The final assertion follows form the second one and Lemma \ref{l:reduce}.

\end{proof}

We close this section with a proof of Theorem \ref{t:graph} and one of its corollaries.

\begin{proof}[{\bf Proof of Theorem \ref{t:graph}}]

Recall that $L=PSL_n^\varepsilon(q)$, $d=(n,q-\varepsilon)$ and consider the difference $\omega(\tau L)\setminus\omega(L)$.
We analyze separately two cases according as  $n$ is odd or even.

Let $n$ be odd. By Lemma \ref{l:graph_odd}, the difference under consideration is equal to $$2\cdot\omega(Sp_{n-1}(q))\setminus\omega(L).$$
We consider the numbers from \ref{l:spec_cn_bn} defining $\omega(Sp_{n-1}(q))$ in turn.

Let $a=2p^t$, where $p^{t-1}+1=n-1$. Any element of $\omega(L)$ that is a multiple of $p^t$
divides $p^t(q-\varepsilon)/d$. Since $d$ is odd, $2a$ divides $p^t(q-\varepsilon)/d$ if and only if $4$ divides $q-\varepsilon$.
Thus $4p^t\in \omega(\tau L)\setminus\omega(L)$ if and only if $n=p^{t-1}+2$ for some $t\geqslant 1$ and $q\equiv -\varepsilon\pmod{4}$.

Let $a=p^t[q^{n_1}\pm 1, q^{n_2}\pm1, \dots, q^{n_s}\pm1]$, where $p^{t-1}+1+ 2n_1 + 2n_2 + \dots + 2n_s = n-1$,
or $a=[q^{n_1}\pm 1, q^{n_2}\pm 1, \dots, q^{n_s}\pm 1]$, where $s\geqslant 2$ and $2(n_1 + n_2 + \dots + n_s) = n-1$.
Then $2a$ divides $p^t[q^{2n_1}-1, q^{2n_2}-1, \dots, q^{2n_s}-1,q-\varepsilon]$ or $[q^{2n_1}-1, q^{2n_2}-1, \dots, q^{2n_s}-1,q-\varepsilon]$ respectively.

The remaining possibility is $a=q^l\pm 1$, where $l=(n-1)/2$. If $a=q^{l}+\varepsilon^l$, then $2a$ divides $$\frac{q^{n-1}-1}{d}=(q^l+\varepsilon^l)\frac{q^l-\varepsilon^l}{d}$$
since $d$ is odd. Let $a=q^l-\varepsilon^l$. Assume that $l$ is not a $2$-power. Then $(l)_2\leqslant l/3$, and hence $l+2(l)_2<n$.
It follows that $L$ has an element of order $$c=[q^l-\varepsilon^l,q^{2(l)_2}-1,q-\varepsilon].$$ Since $2(a)_2=2(q^l-\varepsilon^l)_2\leqslant (q^{2l}-1)_2=(q^{2(l)_2}-1)_2$,
we have that $2a$ divides $c$. Finally, assume that $l$ is a 2-power. Then the 2-exponent of $L$ is equal to $(q^{2l}-1)_2$ and any element of $\omega(L)$ that a multiple of $(q^{2l}-1)_2$
divides $(q^{2l}-1)/d$. Clearly, $2a$ divides $(q^{2l}-1)/d$ if and only if $d$ divides $q^l+\varepsilon^l$, which is equivalent to $d=1$.
Thus $2(q^{l}-\varepsilon^l)\in\omega(\tau L)\setminus\omega(L)$ if and only if $n-1=2^t$ and $d\neq 1$.

Let $n$ be even. By Lemma \ref{l:graph_even}, the set $\omega_{\tilde p}(\tau L)$ is equal to $2\cdot\omega_{\tilde p}(\Omega_{n+1}(q))$ if $n>4$, it consists of divisors of $p(q\pm 1)$ if $n=4$ and $p>3$, and 
it consists of divisors of $p(q\pm 1)$ together with $18$ if $n=4$ and $p=3$. We can consider the numbers from Lemma \ref{l:spec_cn_bn} defining $\omega_{\tilde p}(\Omega_{n+1}(q)$ for $n\geqslant 6$ and 
the numbers defining $\omega_{\tilde p}(\tau L)$ for $n=4$ together.

If $a=p^t$ and $n=p^{t-1}+1$, then $2a\not\in\omega(L)$.

Let $a=p^t[q^{n_1}\pm 1, q^{n_2}\pm 1, \dots, q^{n_s}\pm 1]$, where $s\geqslant 2$ and $p^{t-1}+1+ 2n_1 + 2n_2 + \dots + 2n_s = n$. Then $2a$ divides $p^t[q^{2n_1}-1, q^{2n_2}-1, \dots, q^{2n_s}-1]\in\omega(L)$.

Let $a=p^t(q^{n_1}-\epsilon)/2$, where $p^{t-1}+1+ 2n_1 = n$. If $\epsilon=\varepsilon^{n_1}$, then $2a$ divides $p^t(q^{n_1}-\varepsilon^{n_1})\in\omega(L)$, while if $\epsilon=-\varepsilon^{n_1}$, then 
$2a$ divides $$p^t\frac{q^{2n_1}-1}{d}=p^t(q^{n_1}+\varepsilon^{n_1})\frac{q^{n_1}-\varepsilon^{n_1}}{d}.$$

Thus  $$\omega_{\tilde p}(\tau L)\subseteq \omega(L),$$
whenever $n\neq p^{t-1}+1$.

To handle $\omega_{p'}(\tau L)=2\cdot\omega(P\Omega_{2n}^+(q))\cup 2\cdot \omega(P\Omega_{2n}^-(q))$, we consider the numbers from Lemma \ref{l:spec_dn} that define $\omega_{p'}(P\Omega^\pm_{2n}(q))$.

Let $a=[q^{n_1}\pm 1, q^{n_2}\pm 1, \dots, q^{n_s}\pm 1]$, where $s\geqslant 3$ and $2(n_1 + n_2 + \dots + n_s) = n$. Then $2a$ divides $[q^{2n_1}-1, q^{2n_2}-1, \dots, q^{2n_s}-1]\in \omega(L)$.

Let $a=(q^{n/2}-\epsilon)/(4,q^{n/2}-\epsilon)$. If $\epsilon=\varepsilon^{n/2}$, then $2a$ divides $q^{n/2}-\varepsilon^{n/2}\in\omega(L)$. If $\epsilon=-\varepsilon^{n/2}$ and $(4,q^{n/2}+\varepsilon^{n/2})=4$, then $2a$ divides $$\frac{q^n-1}{(q-\varepsilon)d}=\frac{(q^{n/2}+\varepsilon^{n/2})(q^{n/2}-\varepsilon^{n/2})}{(q-\varepsilon)d}$$
because $d/2$ divides $(n/2,q-\varepsilon)=((q^{n/2}-\varepsilon^{n/2})/(q-\varepsilon),q-\varepsilon)$. If $\epsilon=-\varepsilon^{n/2}$ and $(4,q^{n/2}+\varepsilon^{n/2})=2$, then Lemma \ref{l:2adj}
implies that $2a=q^{n/2}+\varepsilon^{n/2}$ does not lie in $\omega(L)$ if and only if $(n)_2\leqslant (q-\varepsilon)_2$.
Assuming the last inequality, the condition $(4,q^{n/2}+\varepsilon^{n/2})=2$ is equivalent to $q\equiv \varepsilon\pmod 4$.

Finally, let $a=[q^{n_1}-\epsilon_1,q^{n_2}-\epsilon_2]/e$, where $2(n_1+n_2)=n$, $e=2$ if $(q^{n_1}-\epsilon_1)_2=(q^{n_2}-\epsilon_2)_2$ and $e=1$ if $(q^{n_1}-\epsilon_1)_2\neq(q^{n_2}-\epsilon_2)_2$. 
Define $n_0=n/(2n_1,2n_2)$. We may assume that   $(q^{n_1}-\epsilon_1)_2\geqslant (q^{n_2}-\epsilon_2)_2$.

If $\epsilon_2=\varepsilon^{n_2}$, then $2a$ divides $[q^{2n_1}-1,q^{n_2}-\varepsilon^{n_2}]\in\omega(L)$ since $$2(a)_2\leqslant 2(q^{n_1}-\epsilon_1)_2\leqslant (q^{2n_1}-1)_2.$$
Similarly, if $\epsilon_1=\varepsilon^{n_1}$ and $(q^{n_1}-\epsilon_1)_2=(q^{n_2}-\epsilon_2)_2$, then $2a$ divides $[q^{2n_2}-1,q^{n_1}-\varepsilon^{n_1}]$.

Let $\epsilon_2=-\varepsilon^{n_2}$ and $\epsilon_1=-\varepsilon^{n_1}$. If $e=2$, then $2a$ divides $[q^{2n_1}-1,q^{2n_2}-1]/(q-\varepsilon)$. If $e=1$, then $n_1$ and $n_2$ have opposite parity, 
therefore,  $n/2$ is odd. It follows that $n_0$ is odd too, and $2a$ divides $[q^{2n_1}-1,q^{2n_2}-1]/(n_0,q-\varepsilon)$.

Let $\epsilon_2=-\varepsilon^{n_2}$ and $\epsilon_1=\varepsilon^{n_1}$. We may assume that $(q^{n_1}-\varepsilon^{n_1})_2>(q^{n_2}+\varepsilon^{n_2})_2$.  If $n_1\neq (n_1)_2$, then $n_1+2(n_1)_2+2n_2<n$ and $2a$ divides $[q^{n_1}-\varepsilon^{n_1}, q^{2n_2}-1, q^{2(n_1)_2}-1]$.
If $2(q^{n_1}-\varepsilon^{n_1})_2\leqslant (q^{2n_2}-1)_2$, then $2a$ divides $[q^{2n_2}-1,q^{n_1}-\varepsilon^{n_1}]$. If $q^{n_2}+\varepsilon^{n_2}$ divides $q^{2n_1}-1$, then $2a$ divides $q^{2n_1}-1$ too. Finally, if $(n_0,q-\varepsilon)=1$, then $2a$ divides $[q^{2n_1}-1,q^{2n_2}-1]$.

Thus it remains to consider the case when $n_1=(n_1)_2$, $2(q^{n_1}-\varepsilon^{n_1})_2>(q^{2n_2}-1)_2$, $q^{n_2}+\varepsilon^{n_2}$ does not divide $q^{2n_1}-1$ and
$(n_0,q-\varepsilon)\neq 1$. The two first conditions yield $n_1>(n_2)_2$. In particular, $(n/2)_2=(n_2)_2$ and $n_0=n/(2n_2)_2=(n)_{2'}$.
Then the third condition is equivalent to $n_2$ not being a 2-power. Observe that 
$$(n)_{2'}=n/(2n_2)_2=(n_1+n_2)/(n_2)_2=n_1/(n_2)_2+(n_2)_{2'}.$$
It follows that $(n)_{2'}$ is a sum of a non-identity $2$-power  and an odd number greater than 1 and also $((n)_{2'}, q-\varepsilon)\neq 1$. 
In particular, $(n)_{2'}>3$ and $(n,q-\varepsilon)_{2'}\neq 1$.

Conversely, suppose that $n=2^t\cdot l$, where $t\geqslant 1$, $l\geqslant 5$ is odd and $(n,q-\varepsilon)$ is divisible by an odd prime $r$. Writing $n_1=(n)_2=2^t$, $n_2=n/2-n_1=2^{t-1}(l-2)$ and $a=[q^{n_1}-1,q^{n_2}+\varepsilon^{n_2}]$, we see that $$(q^{n_1}-1)_2=2^{t-1}(q^2-1)_2>(n_2)_2(q+\varepsilon)_2\geqslant (q^{n_2}+\varepsilon^{n_2})_2$$ and hence $2a\in\omega(\tau L)$. 
Assume that $2a\in\omega(L)$ and let $c=[q^{l_1}-\varepsilon^{l_1},\dots,q^{l_s}-\varepsilon^{l_s}]/f$, where $l_1+\dots+l_s=n$, be a number from Lemma \ref{l:spec_an} that $2a$ divides.  Since 
$2(a)_2=2(q^{n_1}-1)_2=(q^{2n_1}-1)_2$, some of the numbers $l_1, \dots, l_s$ is a multiple of $2n_1$. Also $r_{2n_2}(\varepsilon q)$ divides $a$, and hence some of them is a multiple of $2n_2$. Observing that
$[2n_1,2n_2]=2^{t+1}(l-2)>2^tl=n$, we deduce that those are different numbers, which yields $s=2$, $l_1=2n_1$ and $l_2=2n_2$. Then $f=(n/(2n_1,2n_2),q-\varepsilon)$, and so $(f)_r>1$. Since $(r,q+\varepsilon)=(r,n_1)=(r,n_2)=1$, it follows that $$(c)_r=\frac{(q-\varepsilon)_r}{(f)_r}<(q-\varepsilon)_r=(a)_r.$$
This is a contradiction, therefore, $2a\not\in\omega(L)$, and the proof is complete.

\end{proof}

Theorem~\ref{t:graph} has an interesting corollary: as the next lemma shows, if $\tau$ is admissible as an automorphism of $PSL^\varepsilon_n(q)$, then it is admissible as an automorphism of $PSL^\varepsilon_n(q^{1/k})$ 
for every odd $k$.

\begin{lemma}\label{l:graph_field} Let $k$ be odd and  $q=q_0^k$. If $\omega(\tau L)\subseteq \omega(L)$, then $$\omega(\tau PSL^\varepsilon_n(q_0))\subseteq\omega(PSL^\varepsilon_n(q_0)).$$
In particular, if $\varepsilon=+$, $\tau$ and $\beta=\varphi^{m/k}$ are admissible for $L$, then $\beta\tau$ is also admissible.
\end{lemma}

\begin{proof}[{\bf Proof}] 
Assume that $\omega(\tau PSL_n^\varepsilon(q_0))\not\subseteq\omega(PSL_n^\varepsilon(q_0))$. This implies that the numbers $n$ and $q_0$ satisfy the conditions of one of the items (i)--(v) of Theorem  \ref{t:graph}.
Since $k$ is odd, it follows that $(n,q_0-\varepsilon)$ divides $(n,q-\varepsilon)$ and $(q-1)_2=(q_0-1)_2$, and thus the numbers $n$ and $q$ satisfy the same conditions.
This contradicts the hypothesis that $\omega(\tau L)\subseteq\omega(L)$.

To prove the second assertion, it suffices to show that $\omega(\beta\tau L)\subseteq\omega(L)$. By Lemma \ref{l:reduce}, the proved containment and admissibility of $\beta$, we see that $$\omega(\beta\tau L)=k\cdot \omega(\tau PSL_n(q_0))\subseteq k\cdot\omega(PSL_n(q_0))=\omega(\beta L)\subseteq\omega(L),$$ and the proof is complete.
\end{proof}

\section{Admissible groups}

In this section, we will prove Theorems \ref{t:main} and \ref{t:mainu}. As we mentioned, the theorems with $n=3$ were proved in \cite{04Zav, 06Zav.t}, and so we assume that $n\geqslant 4$.
Throughout the section, $q=p^m$ is odd, $L=PSL_n^\varepsilon(q)$ and $G$ is a group such that $L<G\leq \Aut L$. Also, we fix the numbers $d=(n,q-\varepsilon)$ and $b=((q-\varepsilon)/d,m)_d$.

We begin with lemmas that holds for both linear and unitary groups. We say that a subgroup of $\Out L$ is admissible if it is the image of an admissible group.

\begin{lemma}\label{l:diag} If $G\cap \Inndiag L>L$, then $\omega(G)\neq\omega(L)$. In particular, admissible groups of $\Out L$ are abelian 
and any non-trivial admissible subgroup of the group $\langle \overline{\delta}\rangle \rtimes \langle \overline{\tau}\rangle$ is conjugate in this group to $\langle \overline\tau\rangle$.
\end{lemma}

\begin{proof}[{\bf Proof}]  By Lemma \ref{l:spec_an}, if $|G\cap \Inndiag L|/|L|=i>1$, then $G$ has an element of order $(q^n-\varepsilon^n)i/(q-\varepsilon)d$, which does not lie in  $\omega(L)$.
Thus admissible groups of $\Out L$ can be embedded into the image of the group generated by $\varphi$ and $\gamma$, which is abelian. The group $\langle \overline{\delta}\rangle \rtimes \langle \overline{\tau}\rangle$
is dihedral, and so every its subgroup that intersects trivially with $\langle \overline{\delta}\rangle$ is conjugate to $\langle \overline\tau\rangle$ or $\langle \overline\delta\overline\tau\rangle$, and in the latter case 
we may assume that $n$ is even. But Lemma \ref{l:graph_diag} says that $\langle \overline\delta\overline\tau\rangle$ is not admissible.
\end{proof}

\begin{lemma}\label{l:oddpart}
Suppose that $n\geqslant 5$ and $|G/L|$ is odd. Then $\omega(G)=\omega(L)$ if and only if $n-1$ is not a $p$-power, $G/L$ is conjugate in $\Out L$ to a subgroup of $\langle\overline\varphi\rangle$ and $|G/L|$ divides $((q-\varepsilon)/d,m)_d$.
\end{lemma}

\begin{proof}[{\bf Proof}] 
If $\varepsilon=+$, the assertion is proved in \cite[Propositions 6,7]{08Gr.t}, and if $\varepsilon=-$, it is proved in \cite[Proposition 6]{13GrShi}.
\end{proof}

\begin{lemma}\label{l:oddpart4}
Let $n=4$ and $|G/L|$ is odd.  Then $\omega(G)=\omega(L)$ if and only if $12$ divides $q+\varepsilon$ and $G/L$ is a $3$-group.
\end{lemma}

\begin{proof}[{\bf Proof}] Let $\omega(G)=\omega(L)$ and $r\in\pi(G/L)$. Then $G$ contains a field automorphism of order $r$, and hence
by Lemma  \ref{l:reduce},  $\omega(G)$ includes $r\cdot\omega(PSL_4^\varepsilon(q_0))$, where $q=q_0^r$. It follows that $rr_4(q_0), rr_3(\varepsilon q_0)\in\omega(L)$.
Since $r_4(q_0)\in R_4(q)$, we see that $r$ divides $a=(q^2+1)(q+\varepsilon)/d$ and, in particular, it does
not divide $q-\varepsilon$. If $r\neq 3$, then $r_3(\varepsilon q_0)\in R_3(\varepsilon q)$, and so $r$ divides $c=(q^3-\varepsilon)/d$, which is a contradiction because $(a,c)=1$.
Let $r=3$. Then $3$ divides $q+\varepsilon$. Since $3p(q_0-\varepsilon)\in\omega(L)$, we have that $3(q_0-\varepsilon)$ divides $(q^2-1)/d$, and thus $4$ divides $q+\varepsilon$.

Conversely, let $12$ divide $q+\varepsilon$ and $G/L$ be a 3-group. We may assume that $G$ is the extension of $L$ by a field automorphism of order dividing $(m)_3$.
To prove that $\omega(G)\subseteq \omega(L)$, it suffices to  check that $k\cdot\omega(PSL_4^\varepsilon(q_0))\subseteq \omega(PSL^\varepsilon_4(q))$ with  $q=q_0^{k}$ for
every divisor $k$ of $(m)_3$. Note that $12$ divides $q_0+\varepsilon$, $(q+\varepsilon)_{3}=k(q_0+\varepsilon)_{3}$ and $(4,q_0-\varepsilon)=2=d$. Since $p\neq 3$, 
the set $\omega(PSL_4^\varepsilon(q_0))$ consists of divisors of $(q_0^2+1)(q_0+\varepsilon)/2$, $(q_0^3-\varepsilon)/2$, $q_0^2-1$ and $p(q_0^2-1)/2$. 
Observing that $$k(q_0^2+1)(q_0+\varepsilon)\mid (q^2+1)(q+\varepsilon),$$ $$k(q_0^3-\varepsilon), k(q_0^2-1)\mid q^2-1,$$
we obtain the desired containment.
\end{proof}

By Lemmas \ref{l:oddpart} and \ref{l:oddpart4}, it follows that up to conjugacy $\Out L$ has  only one maximal admissible subgroup of odd order, and we can take this subgroup to be $\langle \overline\psi\rangle$, where
$\psi\in \langle \varphi\rangle$ has order $(b)_{2'}$ if $n>4$ or $q\not\equiv 1\pmod{12}$, and order $(m)_{3}$ if $n=4$ and $q\equiv -\varepsilon\pmod{12}$.

Let $\eta=\delta^{(d)_{2'}}$ and $S_2$ be the Sylow $2$-subgroup of $\Out L$ generated by $\overline\eta$, $\overline\varphi^{(m)_{2'}}$ and $\overline\tau$.

\begin{lemma}\label{l:direct} If $\omega(G)=\omega(L)$, then $G/L$ is conjugate in $\Out L$ to a subgroup of $\langle \overline \psi\rangle \times S_2$.
\end{lemma}

\begin{proof}[{\bf Proof}]  By Lemma \ref{l:diag},  the group $G/L$ is abelian, and hence is the direct product of its  Hall $2'$-subgroup $A_1$ and Sylow 2-subgroup $A_2$. Since $A_1$ is admissible, it is conjugate to a subgroup of $\langle \overline\psi\rangle$. Replacing $G/L$ by a conjugate if necessary, we may assume that $A_1\leqslant \langle \overline\psi\rangle$.

Note that $\overline\eta$ centralizes $\overline\psi$. Indeed, if $\varepsilon=+$, then $\overline\eta^{\overline\psi}=\overline\eta^{q_0}$, where $q=q_0^{|\psi|}$. 
Since  $|\psi|$ is odd, we have $(q-1)_2=(q_0-1)_2$, and hence $(d)_2$ divides $q_0-1$. If $\varepsilon=-$, then
$\overline\eta^{\overline\psi}=\overline\eta^{q_0}$, where $q^2=q_0^{|\psi|}$. Now $(q+1)_2<(q^2-1)_2=(q_0-1)_2$, and again $(d)_2$ divides $q_0-1$. It follows that the whole group $S_2$ centralizes $\overline\psi$.
Thus $A_2$ is conjugate in $C_{\Out L}(A_1)$ to a subgroup of $S_2$, and the whole group $G/L$ is conjugate to a subgroup of $\langle \overline\psi\rangle\times S_2$.
\end{proof}

The structure of $S_2$ varies according to linear or unitary groups are under consideration, so in the rest of this section we consider the cases $\varepsilon=+$ and $\varepsilon=-$ separately. 
We begin with the case of unitary groups, in which $S_2=\langle \overline\eta\rangle\rtimes \langle \overline\varphi^{(m)_{2'}}\rangle$.

\begin{lemma} \label{l:S2u} Let $\varepsilon=-$ and $1<G/L\leq S_2$. If $\omega(G)=\omega(L)$, then $G/L$ is conjugate to a subgroup of $\langle \overline\varphi^{(m)_{2'}}\rangle$.
\end{lemma}

\begin{proof}[{\bf Proof}] 
We may assume that $n$ is even. If  $m$ is odd, then $S_2=\langle \overline\eta\rangle \rtimes \langle \overline\tau\rangle$ and by Lemma \ref{l:diag}, the group $G/L$ is conjugate to $\langle \overline\tau\rangle$.

Let $m$ be even. Then $|\overline\eta|=2$ and $S_2=\langle \overline\eta\rangle \times \langle \overline\varphi^{(m)_{2'}}\rangle$. Suppose that $G\not\leq \langle \overline\varphi^{(m)_{2'}}\rangle$. Then
$G$ contains $\varphi^{m/k}\eta$ for some $k>1$ dividing $(m)_2$. Let $q=q_0^k$. By Lemmas \ref{l:reduce} and \ref{l:graph_diag} together with Lemma \ref{e:neven}, we have 
$$\omega(\varphi^{m/k}\eta L)=k\cdot \omega(\tau\delta^{d/2}(q_0)PSL_n(q_0))=k\cdot \omega(\tau\delta(q_0)PSL_n(q_0))\not\subseteq\omega(L),$$ which is a contradiction. 
\end{proof}

Thus if $\varepsilon=-$ and $\tau$ is not admissible, then $S_2$ has no non-trivial admissible subgroups. 

\begin{lemma}\label{l:field_u}
Suppose that $\varepsilon=-$, $\tau$ is admissible, $k>1$ divides $(m)_2$ and $\beta=\varphi^{m/k}$. Then 
$\beta$ is admissible if and only if $(n)_2\leqslant 2$ or $n=4,8,12$. If $\beta$ is admissible and $\gamma\in\langle \varphi\rangle$ is admissible and has odd order, then $\gamma\beta$ is also admissible. 
\end{lemma}

\begin{proof}[{\bf Proof}]  We show first that $\omega(\beta L)\subseteq \omega(L)$ if and only if $(n)_2\leqslant 2$ or $n=4,8,12$.
By Lemma \ref{l:reduce}, the set $\omega(\beta L)$ is equal to $k\cdot \omega(\tau PSL_n(q_0))$, where  $q=q_0^k$.

Let $n$ be odd. Since $\tau$ is admissible and $q\equiv 1\pmod 4$, Theorem \ref{t:graph} implies that  $n-2$ is not
a power of $p$. Then using Lemmas \ref{l:spec_cn_bn} and \ref{l:r-part}, it is not hard to check that $k\cdot\omega(Sp_{n-1}(q_0))\subseteq\omega(Sp_{n-1}(q))$.
Now applying Lemma \ref{l:graph_odd} yields 
 $$k\cdot \omega(\tau PSL_n(q_0))=2k\cdot\omega(Sp_{n-1}(q_0))\subseteq 2\cdot\omega(Sp_{n-1}(q))=\omega(\tau L)\subseteq \omega(L).$$

Let $n$ be even. Observe that $(d)_2=(q+1)_2=2$. By admissibility of $\tau$ and Theorem \ref{t:graph}, it follows that $n-1$ is not a $p$-power.
Also it is not hard to verify that $k\cdot\omega(\Omega_{n+1}(q_0))\subseteq\omega(\Omega_{n+1}(q))$ (cf. \cite[Theorem 1]{16Gr.t}).
Applying Lemma \ref{l:graph_even}, we see that $k\cdot \omega_{\tilde p}(\tau PSL_n(q_0))\subseteq \omega_{\tilde p}(\tau L)$ for $n>4$. Since $k(q_0\pm 1)$ divides $q-1$, the same is true for $n=4$.
Thus it remains to examine when $2k\cdot \omega_{p'}(P\Omega_n^\pm (q_0))\subseteq \omega(L))$, and we consider the numbers from Lemma \ref{l:spec_dn} in turn.

Let $a=[q_0^{n_1}\pm 1, q_0^{n_2}\pm 1, \dots, q_0^{n_s}\pm 1]$, where $s\geqslant 3$ and $2(n_1 + n_2 + \dots + n_s) = n$. Then $2ka$ divides $[q^{2n_1}-1, q^{2n_2}-1, \dots, q^{2n_s}-1]\in \omega(L)$.

Let $a=(q_0^{n/2}-\epsilon)/(4,q_0^{n/2}-\epsilon)$. If $n/2$ is even, then $2ka$ divides $q^{n/2}-1\in\omega(L)$.
Suppose that $n/2$ is odd. Then $r_{n/2}(\epsilon q)$ divides $a$ and lies in $R_{n/2}(q)=R_n(-q)$, and so $2ka\in \omega(L)$ id and only if 
$2ka$ divides $$c=\frac{q^n-1}{(q+1)d}=\frac{(q^{n/2}-1)(q^{n/2}+1)}{(q+1)d}.$$ Clearly, $(a)_{2'}$ divides $(c)_{2'}$. Since $(d)_2=2$, we see that $(c)_2=(q-1)_2/2$.
If $(4,q_0-\epsilon)=4$, then $2k(a)_2=k(q_0-\epsilon)_2/2=(q-1)_2/2.$ If $(4,q_0-\epsilon)=2$, then $2k(a)_2=2k\leqslant (q-1)_2/2$.

Finally, let $a=[q_0^{n_1}-\epsilon_1,q_0^{n_2}-\epsilon_2]/e$, where $2(n_1+n_2)=n$, $e=2$ if $(q^{n_1}-\epsilon_1)_2=(q^{n_2}-\epsilon_2)_2$ and $e=1$ if $(q^{n_1}-\epsilon_1)_2\neq(q^{n_2}-\epsilon_2)_2$.

Suppose that $n/2$ is odd. We may assume that $(q^{n_1}-\epsilon_1)_2\geqslant (q^{n_2}-\epsilon_2)_2$.
If $n_2$ is even, then $2ka$ divides $[q^{2n_1}-1,q^{n_2}-1]$. Let $n_2$ is odd. Then $n_1$ is even. If $n_1\neq (n_1)_2$, then
$2ka$ divides $[q^{n_1}-1,q^{(n_1)_2}-1,q^{2n_2}-1]$. If $n_1=(n_1)_2$ and $n_2=1$, then $2ka$ divides $q^{2n_1}-1$. If $(n_1)=(n_1)_2$ and $n_2>1$, then
by admissibility of $\tau$, we have $(n,q+1)_{2'}=1$ and, therefore, $(n/(2n_1,2n_2),q+1)=1$. So $[q^{2n_1}-1,q^{2n_2}-1]$ lies in $\omega(L)$, and it is divisible by $2ka$.

Suppose that $n/2$ is even and $n/2\geqslant 8$. We can take $n_1$ and $n_2$ to be odd coprime numbers larger than 1. 
Also we take $\epsilon_1=+$ and $\epsilon_2=-$. Then $a=[q_0^{n_1}]-1,q_0^{n_2}+1]$ and $2k(a)_2=k(q_0^2-1)_2=(q^2-1)_2$. Since $a$ is a multiple of 
both $r_{n_1}(q_0)$ and $r_{n_2}(-q_0)$ and $r_{n_i}(\pm q_0)\in R_{2n_i}(-q)$ for $i=1,2$, it follows that $2ka\in \omega(L)$ if and only if $2ka$ divides $c=[q_{2n_1}-1,q^{2n_2}-1]/(n/(2n_1,2n_2),q+1)$.
But $(c)_2=(q^2-1)_2/2<(a)_2$ because $n/2$ is even. Thus $\omega(\beta L)\not\subseteq\omega(L)$.

We are left with the cases $n=4,8,12$. If $n=4$, then $2a$ divides $q_0^2-1$, and hence $2ka$ divides $q^2-1$.
If $n=8$, then $a$ divides $[q_0^3\pm 1,q_0\pm 1]$ or $(q_0^4-1)/2$, and so $2ka$ divides $q^6-1$ or $q^4-1$.
If $n=12$, then $a$ divides $[q_0^5\pm 1,q_0\pm 1]$, $[q_0^4\pm 1,q_0^2\pm 1]$, or $(q_0^6-1)/2$, and $2ka$ divides $q^{10}-1$, $q^8-1$, or $q^6-1$.

We established that the condition for $\omega(\beta L)$ to be a subset of $\omega(L)$ depends only on $n$ and not on $q_0$, and thus $\beta$ is admissible if and only if $\omega(\beta L)\subseteq\omega(L)$.

Next, suppose that $\beta$ is admissible, $\gamma\in\langle \varphi\rangle$ is admissible and has odd order $l$ and let $q_0=q_1^l$. To prove that $\gamma\beta$ is admissible, it suffices to show that 
$\omega(\gamma\beta L)\subseteq \omega(L)$.  By Lemma \ref{l:graph_field}, if we regard $\tau$ as an automorphism of $PSU_n(q_1^k)$, it is still be admissible, and so by the above result, it follows that
$k\cdot\omega(\tau PSL_n(q_1))\subseteq \omega (PSU_n(q_1^k))$. By Lemma \ref{l:reduce}, we have
$$\omega(\gamma\beta L)=lk\cdot\omega(\tau PSL_n(q_1))\subseteq l\cdot \omega (PSU_n(q_1^k))=\omega(\gamma L)\subseteq \omega(L).$$
\end{proof}

We are in position to prove Theorem Theorem \ref{t:mainu}.

\begin{proof}[{\bf Proof of Theorem \ref{t:mainu}}] Suppose that $\omega(G)=\omega(L)$. By Lemmas \ref{l:direct} and \ref{l:S2u}, it follows that $G/L$ is conjugate to a subgroup in $\langle \overline\psi\rangle\times \langle \overline\varphi^{(m)_{2'}}\rangle$.
If $n-1$ is a $p$-power, then Lemmas \ref{l:graph_even}, \ref{l:oddpart} and \ref{l:oddpart4} implies that any non-trivial element of $G/L$ is not admissible. If $n-1$ is not a $p$-power, then
$\psi$ is admissible, and applying Lemma \ref{l:field_u} completes the proof.

\end{proof}

Now we assume that $\varepsilon=+$, and so $S_2=\langle \overline\eta\rangle\rtimes(\langle\overline\varphi^{m_{2'}}\rangle\times \langle\overline \tau\rangle)$.
Given $1\neq\beta\in\langle \varphi^{(m)_{2'}}\rangle$, we define $\beta_\epsilon$ with $\epsilon\in\{+,-\}$ by setting $\beta_+=\beta$ and $\beta_-=\beta\tau$.

\begin{lemma}\label{l:evenfield}
Let $\varepsilon =+$, $k>1$ divides $(m)_2$ and $q=q_0^k$. Then for any $\epsilon\in\{+,-\}$, the following hold:
\begin{enumerate}
 \item if $n$ is odd or $k$ does not divide $(q-1)/d$, then $k\cdot\omega(PSL_n^\epsilon(q_0))\not\subseteq \omega(L)$;
 \item if $n$ is even, $k$ divides $(q-1)/d$ and $n=1+p^{t-1}$, then $k\cdot\omega(PSL_n^\epsilon(q_0))\setminus\{kp^t\}\subseteq\omega(L)$;
 \item if $n$ is even, $k$ divides $(q-1)/d$ and $n-1$ is not a $p$-power, then $k\cdot\omega(PSL_n^\epsilon(q_0))\subseteq\omega(L)$.
\end{enumerate}
In particular, if $\beta=\varphi^{m/k}$, then $\beta_\epsilon$ is admissible if and only if $n$ is even,  $k$ divides $(q-1)/d$ and $n-1$ is not a $p$-power.
If $\beta_\epsilon$ is admissible and $\gamma\in \langle \varphi\rangle$ is admissible and has odd order, then $\gamma\beta_\epsilon$ is also admissible.
\end{lemma}

\begin{proof}

If $n$ is odd, then $2r_n(\epsilon q_0)\not\in\omega(L)$ since $r_n(\epsilon q_0)\in R_n(q)$
and $(q^n-1)/(q-1)$ is odd. If $n$ is even and $k$ does not divide $(q-1)/d$, then $kr_{n-1}(\epsilon q_0)\not\in\omega(L)$. Indeed, otherwise the fact that $r_{n-1}(\epsilon q_0)\in R_{n-1}(q)$ implies that 
$k$ divides $(q^{n-1}-1)/d$. But $(q^{n-1}-1)_2/(d)_2=(q-1)_2/(d)_2$, and (i) follows.

Let $n$ be even and $k$ divides $(q-1)/d$. We claim that $ka\in\omega(L)$ for all  $a\in\omega(PSL_n^\epsilon q_0))$ except $a=p^t$ for $n=1+p^{t-1}$.
The number $a$ divides one of the numbers in items (i)--(v) of Lemma \ref{l:spec_an}, and we consider these possibilities in turn. 

If $a=(q_0^n-1)/(q_0-\epsilon)(n,q_0-\epsilon)$, then $ka$ divides $q^{n/2}-1$. Similarly, if $a=[q_0^{n_1}-\epsilon^{n_1}, q_0^{n_2}-\epsilon^{n_2}]/(n/(n_1,n_2),q_0-\epsilon]$,
where $n_1+n_2=n$ and both $n_1$, $n_2$ are even, then $ka$ divides $[q^{n_1/2}-1,q^{n_2/2}-1]$. If $a=[q^{n_1}-\epsilon^{n_1}, q^{n_2}-\epsilon^{n_2}]/(n/(n_1,n_2),q_0-\epsilon]$, where $n_1$ and $n_2$  are odd,
then $ka$ divides $c=[q^{n_1}-1, q^{n_2}-1]/(n/(n_1,n_2),q-1]$. Indeed, if $q_0\equiv -\epsilon\pmod 4$, then $k(a)_2=k\leqslant (c)_2$, while if $q_0\equiv \epsilon\pmod 4$, then $k(q_0-\epsilon)_2=(q-1)_2\geqslant (k)_2(n)_2$,
which implies that $(q_0-\epsilon)_2\geqslant (n)_2$ and $k(a)_2=(q-1)_2/(n)_2=(c)_2$. Also, $a$ divides $c$ by Lemma \ref{l:gcd}.
Finally, if $a=[q_0^{n_1}-\epsilon^{n_1},\dots,q_0^{n_s}-\epsilon^{n_s}]$, where $s\geqslant 3$ and $n_1+\dots,+n_s=n$,
or $a=p^t[q_0^{n_1}-\epsilon^{n_1},\dots,q_0^{n_s}-\epsilon^{n_s}]$, where $s\geqslant 2$ and $1+p^{t-1}+n_1+\dots,+n_s=n$, then $ka$ divides $[q^{n_1}-1,\dots,q^{n_s}-1]$ or $p^t[q^{n_1}-1,\dots,q^{n_s}-1]$,
respectively. 

Let $\beta=\varphi^{m/k}$. By Lemma \ref{l:reduce}, we have $\omega(\beta_\epsilon L)=k\cdot\omega(PSL_n^\epsilon(q_0))$. Thus if
$\beta_\epsilon$ is admissible, then $n$ is even, $k$ divides $(q-1)/d$ and $n-1$ is not a $p$-power. Conversely, if all these three conditions are satisfied, then by the above 
$$\omega(\beta^{k_1}L)=(k/k_1)\cdot \omega(PSL_n(q_0^{k_1}))\subseteq\omega(L)$$ for every divisor $k_1$ of $k$, and hence $\beta_\epsilon$ is admissible. 

To prove the final assertion, it suffices to check that  $\omega(\gamma\beta_\epsilon L)\subseteq \omega(L)$ for admissible $\gamma$ and $\beta_\epsilon$.
Let $|\gamma|=l$ and $q_0=q_1^l$. Observing that $(q-1)_2=(q_1^{k}-1)_2$, we see that $k$ divides $(q-1)/d$ if and only if 
it divides $(q_1^{k}-1)/(n,q_1^{k}-1)$. Thus admissibility of $\beta_\epsilon$ yields $k\cdot \omega(PSL_n^\epsilon(q_1))\subseteq \omega(PSL_n(q_1^{k}))$. Applying Lemma \ref{l:reduce},
we have $$\omega(\gamma\beta_\epsilon L)=lk\cdot\omega(PSL_n^\epsilon(q_1))\subseteq l\cdot\omega(PSL_n(q_1^{k}))=\omega(\gamma L)\subseteq\omega(L),$$
and the proof is complete.
\end{proof}

\begin{lemma}\label{l:evendfield} Let $\varepsilon=+$, $n$ be even, $k>1$ divide $(m)_2$, $\beta=\varphi^{m/k}$ and $q=q_0^k$.
If $\alpha=\beta_\epsilon\eta^j$ is admissible, then $\overline\alpha$ is conjugate in $S_2$ to either $\overline\beta_\epsilon$ or $\overline\beta_\epsilon\overline\eta$,
with $q_0\equiv -\epsilon\pmod 4$ and $2k\leqslant (q-1)_2/(d)_2$ in the latter case. Furthermore, if $q_0\equiv -\epsilon\pmod 4$ and $2k\leqslant (q-1)_2/(d)_2$, then
\begin{enumerate}
 \item $\beta_\epsilon\eta$ is admissible if and only if $n$ cannot be represented as $1+p^{t-1}+2^u$ with $t,u\geqslant 1$ and, in addition, $k=2$ whenever
 $n-1$ is a $p$-power;
 \item if $\beta_\epsilon\eta$ is admissible and $\gamma\in\langle \varphi\rangle$ is admissible and has odd order, then $\gamma\beta_\epsilon\eta$ is admissible. 
\end{enumerate}
\end{lemma}

\begin{proof}[{\bf Proof}]  Note that $\overline\eta^{\overline\beta_\epsilon}=\overline\eta^{\epsilon q_0}$, and so 
\begin{equation}\label{e:conj}\overline\beta_\epsilon^{\overline\eta}=\overline\eta^{-1}\overline\beta_\epsilon\overline\eta=\overline\beta_\epsilon\overline\eta^{1-\epsilon q_0}.\end{equation}

Suppose that $\alpha=\beta_\epsilon\eta^j$ is admissible and $\overline\alpha$ is not conjugate in $S_2$ to  $\overline\beta_\epsilon$. Then (\ref{e:conj}) implies that $(q_0-\epsilon, (d)_2)$ does not divide $j$, or in other words  \begin{equation}\label{e:j}(j)_2<(q_0-\epsilon,d)_2.\end{equation}
Since $\eta=\delta^{(d)_{2'}}$, applying Lemma \ref{l:reduce} yields  \begin{equation}\label{e:spectrum}\omega(\alpha L)=k\cdot \omega(\delta^{j(d)_{2'}}(\epsilon q_0)PSL_n^\epsilon(q_0)).\end{equation}

The index of $\delta^{j(d)_{2'}}(\epsilon q_0)PSL_n^\epsilon(q_0)$ in $PGL_n^\epsilon(q_0)$ is equal to $(j(d)_{2'},q_0-\epsilon)$. There is an element of order $q_0^{n-1}-\epsilon$ in $PGL_n(q_0)$, and  hence writing $a=(q_0^{n-1}-\epsilon)/(j(d)_{2'},q-\epsilon)$, we see that $ka\in\omega(\alpha L)$. Since $a$ is a multiple of $r_{n-1}(\epsilon q_0)\in R_{n-1}(q)$ and lies in $\omega(L)$, it follows that
$a$ divides $c=(q^{n-1}-1)/d$. If $q_0\equiv \epsilon \pmod 4$, then by (\ref{e:j}), we deduce that $$(a)_2=k(q_0-\epsilon)_2/(j)_2=(q-1)_2/(j)_2>(q-1)_2/(d)_2=(c)_2.$$
Thus $q_0\equiv -\epsilon \pmod 4$. Then $(q_0-\epsilon,d)_2=2$, and so $j$ is odd. In particular, (\ref{e:conj}) shows that $\alpha$ is conjugate to $\beta_\epsilon\eta$.
Also, $(a)_2=2k\leqslant (c)_2=(q-1)_2/(d)_2$, and the first assertion is proved. 

Next, let $\alpha=\beta_\epsilon\eta$, $q_0\equiv -\epsilon\pmod 4$ and $2k\leqslant (q-1)_2/(d)_2$. Observing that 
$(q-1)_2=k(q_0+\epsilon)_2$, we deduce from the last inequality that $2(d)_2\leqslant (q_0+\epsilon)_2$, and hence $\overline\alpha^2=\overline\beta^2\overline\eta^{1+\epsilon q_0}=\overline\beta^2$.
Since $2k\leqslant (q-1)_2/(d)_2$, it follows from Lemma \ref{l:evenfield} that $\beta^2$ is admissible if and only if $k=2$ or $n-1$ is not a $p$-power.

Suppose that $n=1+p^{t-1}+2^u$ with $t,u\geqslant 1$. Then $PGL_n^\epsilon(q_0)$ has an element of order $p^t(q_0^{2^u}-1)$, and so  $ka\in\omega (\alpha L)$, where $a=p^t(q_0^{2^u}-1)/(d,q_0-\epsilon)_{2'}$.
Since $ka$ is a multiple of $kp^t(q_0^{2^u}-1)_2=p^t(q^{2^u}-1)_2$, it lies in $\omega(L)$ only if it divides  $p^t(q^{2^u}-1)/d$, which is not the case. 

To prove (i), it remains to verify that $\omega(\alpha L)\subseteq \omega (L)$ whenever $n$ cannot be represented as $1+p^{t-1}+2^u$ with $t,u\geqslant 1$. SInce 
$((d)_{2'},q_0-\epsilon)=(n,q_0-\epsilon)_{2'}=(n,q_0-\epsilon)/2$, it follows from (\ref{e:spectrum}) that $$\omega(\alpha L)=k\cdot \omega (2PSL_n^\epsilon(q_0)\setminus PSL_n^\epsilon(q_0)),$$
where $2PSL_n^\epsilon(q_0)$ denotes the unique subgroup of $PGL_n^\epsilon(q_0)$ that contains $PSL_n(q_0)$ as a subgroup of index 2. Denote $\omega (2PSL_n^\epsilon(q_0)\setminus PSL_n(q_0))$
by $\omega$.

If $n=p^{t-1}+1$, then any element of $GL_n(q)$ whose order is divisible by $p^t$ is conjugate to a scalar multiple of the $n\times n$ unipotent Jordan block, therefore, 
$p^t\not\in\omega$. By Lemma \ref{l:evenfield}, we have $$k\cdot(\omega\cap \omega(PSL_n^\epsilon(q_0)))\subseteq \omega(L),$$ and hence we are left with elements of $\omega\setminus\omega(PSL_n^\epsilon(q_0))$.
By Lemma \ref{l:spec_an}, it follows that this difference consists of some divisors of the following numbers:
\begin{gather*}
2(q_0^n-1)/((q_0-\epsilon)(n,q_0-\epsilon)),\\
[q_0^{n_1}-\epsilon^{n_1}, q^{n_2}-\epsilon^{n_2}]/(n/(n_1,n_2), q_0-\epsilon), \text{ where  $n_1,n_2$ are odd and } n_1 + n_2= n; \\
2p^t(q_0^{n_1}-\epsilon^{n_1})/(n,q_0-\epsilon), \text{ where $n_1>0$ and } p^{t-1}+1+n_1=n.
\end{gather*}

Let $a=2(q_0^n-1)/((q_0-\epsilon)(n,q_0-\epsilon))$. Then
$$k(a)_2=2k\frac{(q_0^n-1)_2}{(q_0-\epsilon)_2(q_0-\epsilon,n)_2}=(q^n-1)_2/2=(q^{n/2}-1)_2,$$
and so $ka$ divides $q^{n/2}-1\in \omega(L)$.

Let $a=2[q_0^{n_1}-\epsilon^{n_1}, q^{n_2}-\epsilon^{n_2}]/(n/(n_1,n_2), q_0-\epsilon)$, where $n_1$, $n_2$ are odd and $n_1+n_2=n$. Consider the number $c=[q^{n_1}-1,q^{n_2}-1]/(n/(n_1,n_2),q-1)$ lying in $\omega(L)$.
By Lemma \ref{l:gcd},  both numbers $(q^{n_i}-1)/(n/(n_1,n_2), q_0-\epsilon)$, where $i=1,2$, divide $c$. Also, $k(a)_2=2k\leqslant (q-1)_2/d_2=(c)_2$, and hence $ka$ divides  $c$.

Let $a=2p^t(q_0^{n_1}-\epsilon^{n_1})/(n,q_0-\epsilon)$, where $n_1>0$ and $p^{t-1}+1+n_1=n$. Note that $n_1$ is even and by hypothesis, $(n_1)_2\neq n_1$. Since 
$$k(a)_2=2k(q_0^{n_1}-1)_2/(q_0-\epsilon)_2=(q^{n_1}-1)_2=(q^{(n_1)_2}-1)_2,$$ we see that $ka$ divides $p^t[q^{n_1}/2-1, q^{(n_1)_2}-1]\in\omega(L)$.

To prove (ii), take admissible $\gamma$ and $\alpha$. We may assume that $n-1$ is not a $p$-power because otherwise $\gamma=1$ by Lemmas \ref{l:oddpart} and \ref{l:oddpart4}.
Then by the equality $\overline\alpha^2=\overline\beta^2$ and Lemma \ref{l:evendfield}, to prove that $\gamma\alpha$ is admissible, it suffices to check that $\omega(\gamma\alpha L)\subseteq \omega(L)$.
Denote $|\gamma|$ by $l$ and let $q_0=q_1^l$. Observe that $q_1\equiv q_0\equiv -\epsilon\pmod 4$ and $2k$ divides $(q_1^k-1)_2/(n,q_1^k-1)_2$.

By Lemma \ref{l:reduce}, we have $$\omega(\gamma\alpha L)=lk\cdot \omega(\delta^{(d)_{2'}}(\epsilon q_1)PSL^\epsilon_n(q_1)).$$
Using the fact that $(d,q_1-\epsilon)_{2'}=(q_1-\epsilon)/2$ and the above results, we deduce that $$k\cdot\omega(\delta^{(d)_{2'}}(\epsilon q_1)PSL^\epsilon_n(q_1))=k\cdot\omega(2PSL_n^\epsilon(q_1)\setminus PSL_n^\epsilon(q_1))\subseteq \omega(PSL_n(q_1^k)).$$
Thus 
$$\omega(\gamma\alpha L)\subseteq l\cdot\omega(PSL_n(q_1^k)))=\omega(\gamma L)\subseteq \omega(L),$$
and the proof is complete.
\end{proof}

We are now ready to describe admissible 2-subgroups of $\Out L$ in the case of linear groups, and then prove Theorem \ref{t:main}. Recall that $b=((q-\varepsilon)/d,m)_d$ and define $\phi=\varphi^{m/(b)_2}$.

\begin{lemma}\label{l:S2l} Let $\varepsilon=+$ and $1<G/L\leq S_2=\langle \overline \eta\rangle\rtimes (\langle \overline\varphi^{(m)_{2'}}\rangle\times \langle \overline\tau\rangle)$. Then
$\omega(G)=\omega(L)$ if and only if $G/L$ is cyclic and up to conjugacy in $S_2$ is generated by the image of one of the following elements:
\begin{enumerate}
 \item $\tau$ if $b$ is odd and $\tau$ is admissible;
 \item $\beta_\pm$, where $1\neq\beta\in\langle \phi\rangle$, if $b$ is even and $n-1$ is not a $p$-power;
 \item $\beta\tau\eta$, where  $1\neq\beta\in\langle \phi^2\rangle$, if $(b)_2>2$ and $n$ cannot be represented as $1+p^{t-1}$ or $1+p^{t-1}+2^u$ with $t,u\geqslant 1$;
 \item $\phi_{-\kappa}\eta$, where $p\equiv \kappa\pmod 4$, if  $(p-\kappa)_2>(n)_2$ and  $n$ cannot be represented as $1+p^{t-1}$ or $1+p^{t-1}+2^u$ with $t,u\geqslant 1$;
 \item $\varphi^{m/2}\tau\eta$, if $(b)_2>2$, $n=1+p^{t-1}$ for some $t\geqslant 1$ and $n$ cannot be represented as $2+2^u$ with $u\geqslant 1$;
 \item $(\varphi^{m/2})_{-\kappa}\eta$, where $p\equiv\kappa\pmod 4$, if $(m)_2=2$, $(p-\kappa)_2>(n)_2$, $n=1+p^{t-1}$ for some $t\geqslant 1$ and $n$ cannot be represented as $2+2^u$ with $u\geqslant 1$.
\end{enumerate}
\end{lemma}

\begin{proof}[{\bf Proof}] 
Observe that $b$ is even if and only if all the numbers $n$, $m$ and $(q-1)/d$ are even.

Suppose that $n$ is odd. Then $S_2=\langle \overline\varphi^{(m)_{2'}}\rangle \times \langle \overline\tau\rangle$. By Lemma \ref{l:evenfield}, any non-trivial element of $\langle \varphi^{(m)_{2'}}\rangle$ is not admissible.
and thus the only potentially admissible subgroup of $S_2$ is $\langle\overline\tau\rangle$.

Suppose that $n$ is even but at least one of $m$ and $(q-1)/d$ is odd. It follows from Lemmas \ref{l:evenfield} and \ref{l:evendfield} that every admissible subgroup of $S_2$ is contained in $\langle \overline \eta\rangle\rtimes \langle \overline\tau\rangle$. Then by Lemma \ref{l:diag}, it is conjugate to a subgroup of $\langle\overline\tau\rangle$, and we are done in the case of odd $b$.

Suppose now that all the numbers $n$, $m$ and $(q-1)/d$ are even. In particular,  $(q-1)_2>(d)_2=(n)_2$ and so $\tau$ is not admissible by Theorem \ref{l:graph_even}. Then by Lemma \ref{l:diag}, we obtain that all non-trivial elements of  $\langle \overline \eta\rangle\rtimes \langle \overline\tau\rangle$ are not admissible. It follows that every admissible subgroup of $S_2$ is isomorphic to a subgroup of $\langle \overline\varphi^{(m)_{2'}}\rangle$, 
and hence is cyclic.

Let $\overline\alpha\in S_2\setminus \langle \overline \eta\rangle\rtimes \langle \overline\tau\rangle$ is admissible. We may assume that $\alpha=\beta_\epsilon\eta^j$, where
$\beta=\varphi^{m/k}$ for some $1<k=2^l\leqslant (m)_2$. By Lemmas  \ref{l:evenfield} and \ref{l:evendfield}, admissibility of $\alpha$ yields $k\leqslant (q-1)_2/(d)_2$ or, equivalently, $\beta\in\langle \phi\rangle$.
Moreover, all conjugates of $\overline\beta_\epsilon$ are admissible whenever $n-1$ is not a $p$-power, and this gives the automorphisms in (ii).

It remains to consider the case when $\overline\alpha$ is not conjugate to $\overline\beta_\epsilon$.  In this case by Lemma \ref{l:evendfield}, we may assume that $\alpha=\beta_\epsilon\eta$, \begin{equation}\label{e:2k}2k\leqslant (q-1)_2/(d)_2\end{equation} and $q_0\equiv -\epsilon\pmod 4$, where $q=q_0^k$. Moreover, $k=2$ whenever $n-1$ is a $p$-power. Note that $(q-1)_2=(m)_2(p-\kappa)_2$, where $p\equiv \kappa\pmod 4$.

Suppose that $(q-1)_2/(d)_2\leqslant (m)_2$ or, in other words, $(p-\kappa)_2\leqslant (n)_2$. Then $k\leqslant(m)_2/2$, so $q_0\equiv 1\pmod 4$ and the condition (\ref{e:2k}) is equivalent to $\beta\in\langle \phi^2\rangle$. Thus  $(b)_2>2$ and $\alpha=\beta\tau\eta$ with $\beta\in\langle \phi^2\rangle$. In particular, if $n-1$ is a $p$-power, then $\alpha=\varphi^{m/2}\tau\eta$.

Suppose that $(q-1)_2/(d)_2>(m)_2$ or, equivalently, $(p-\kappa)_2>(n)_2$. Then the condition (\ref{e:2k}) holds for all $\beta\in \langle\phi\rangle$. Also, $q_0\equiv 1\pmod 4$ if $k<(m)_2$ and $q_0\equiv \kappa\pmod 4$ if $k=(m)_2$. Thus $\alpha=\beta\tau\eta$ with $\beta\in\langle \phi^2\rangle$ or $\alpha= \phi_{-\kappa}\eta$. In particular, if $n-1$ is a $p$-power, then $\alpha=\varphi^{m/2}\tau\eta$ when $(b)_2=(m)_2>2$ and $\alpha=\phi_{-\kappa}\eta$ when $(b)_2=(m_2)=2$.

Conversely, let $\alpha$ be one of the resulting automorphisms and let the associated conditions be satisfied. Then  by Lemma \ref{l:evendfield}, $\alpha$ is admissible unless $n$ cannot be represented as $1+p^{t-1}+2^u$.
If $n=1+p^{t-1}$ and $n=1+p^{s-1}+2^u$, then $p^{t-1}=p^{s-1}+2^u$ or, equivalently, $p^{s-1}(p^{t-s}-1)=2^u$, and so $s=1$ and $n=2+2^u$. Thus we obtain the automorphisms in (iii)--(vi), and the proof is complete.
\end{proof}

\begin{proof}[{\bf Proof of Theorem \ref{t:main}}] Let $\omega(G)=\omega(L)$. By Lemma \ref{l:direct}, it follows that $G/L$ is conjugate to a subgroup of $\langle \overline\psi\rangle\times S_2$. Admissible subgroups of $\langle \overline\psi\rangle$ and $S_2$ are described in Lemmas \ref{l:oddpart}, \ref{l:oddpart4} and Lemma \ref{l:S2l} respectively. These descriptions together with Lemmas \ref{l:evenfield} and \ref{l:evendfield} imply that the product of an admissible subgroup of $\langle \overline\psi\rangle$ and an admissible subgroup of $S_2$ is also admissible, and this completes the proof. 
\end{proof}


\begin{thebibliography}{10}

\bibitem{94BrShi}
R.~Brandl and W.J. Shi, {The characterization of $PSL(2,q)$ by its element
  orders}, \emph{J. Algebra} \textbf{163} (1994), no.~1, 109--114.

\bibitem{13BHRD}
J.~Bray, D.~Holt, and C.~Roney-Dougal, {The maximal subgroups of the
  low-dimensional finite classical groups}, {London Mathematical Society
  Lecture Note Series}, vol. 407, Cambridge University Press, Cambridge, 2013.

\bibitem{08But.t}
A.~A. Buturlakin, {Spectra of finite linear and unitary groups}, \emph{Algebra
  Logic} \textbf{47} (2008), no.~2, 91--99.

\bibitem{10But.t}
A.~A. Buturlakin, {Spectra of finite symplectic and orthogonal groups},
  \emph{Siberian Adv. Math.} \textbf{21} (2011), no.~3, 176--210.

\bibitem{07ButGr.t}
A.~A. Buturlakin and M.~A. Grechkoseeva, {The cyclic structure of maximal tori
  of the finite classical groups}, \emph{Algebra Logic} \textbf{46} (2007),
  no.~2, 73--89.

\bibitem{85Atlas}
J.~H. Conway, R.~T. Curtis, S.~P. Norton, R.~A. Parker, and R.~A. Wilson,
  {Atlas of finite groups}, Clarendon Press, Oxford, 1985.

\bibitem{12FulGur}
J.~Fulman and R.~Guralnick, {Bounds on the number and sizes of conjugacy
  classes in finite Chevalley groups with applications to derangements.},
  \emph{{Trans. Amer. Math. Soc.}} \textbf{364} (2012), no.~6, 3023--3070.

\bibitem{04FulGur}
Jason Fulman and Robert Guralnick, Conjugacy class properties of the extension
  of {${\rm GL}(n,q)$} generated by the inverse transpose involution, \emph{J.
  Algebra} \textbf{275} (2004), no.~1, 356--396.

\bibitem{08GowVin}
Rod Gow and C.~Ryan Vinroot, Extending real-valued characters of finite general
  linear and unitary groups on elements related to regular unipotents, \emph{J.
  Group Theory} \textbf{11} (2008), no.~3, 299--331.

\bibitem{08Gr.t}
M.~A. Grechkoseeva, {Recognition by spectrum for finite linear groups over
  fields of characteristic $2$}, \emph{Algebra Logic} \textbf{47} (2008),
  no.~4, 229--241.

\bibitem{16Gr.t}
M.~A. Grechkoseeva, {On spectra of almost simple groups with symplectic or
  orthogonal socle}, \emph{Siberian Math. J.} \textbf{57} (2016), no.~4,
  582--588.

\bibitem{13GrShi}
M.~A. Grechkoseeva and W.~J. Shi, {On finite groups isospectral to finite
  simple unitary groups over fields of characteristic 2}, \emph{Sib.
  {\'E}lektron. Mat. Izv.} \textbf{10} (2013), 31--37.

\bibitem{15VasGr1}
M.~A. Grechkoseeva and A.~V. Vasil'ev, {On the structure of finite groups
  isospectral to finite simple groups}, \emph{J. Group Theory} \textbf{18}
  (2015), no.~5, 741--759.

\bibitem{90KlLie}
P.~Kleidman and M.~Liebeck, {The subgroup structure of the finite classical
  groups}, {London Mathematical Society Lecture Note Series}, vol. 129,
  Cambridge University Press, Cambridge, 1990.

\bibitem{94Maz.t}
V.~D. Mazurov, On the set of orders of elements of a finite group,
  \emph{Algebra and Logic} \textbf{33} (1994), no.~1, 49--55.

\bibitem{09Tay}
D.~E. Taylor, {The geometry of the classical groups}, {Sigma Series in Pure
  Mathematics}, vol.~9, Heldermann Verlag, Lemgo, 2009.

\bibitem{15Vas}
A.~V. Vasil'ev, {On finite groups isospectral to simple classical groups},
  \emph{J. Algebra} \textbf{423} (2015), 318--374.

\bibitem{63Wall}
G.~E. Wall, {On the conjugacy classes in the unitary, symplectic and orthogonal
  groups}, \emph{J. Aust. Math. Soc.} \textbf{3} (1963), 1--62.

\bibitem{04Zav}
A.~V. Zavarnitsine, {Recognition of the simple groups $L_3(q)$ by element
  orders}, \emph{J. Group Theory} \textbf{7} (2004), no.~1, 81--97.

\bibitem{06Zav.t}
A.~V. Zavarnitsine, {Recognition of the simple groups $U_3(q)$ by element
  orders}, \emph{Algebra Logic} \textbf{45} (2006), no.~2, 106--116.

\bibitem{Zs}
K.~Zsigmondy, {Zur Theorie der Potenzreste}, \emph{Monatsh. Math. Phys.}
  \textbf{3} (1892), 265--284.

\bibitem{16Zve.t}
M.~A. Zvezdina, {On spectra of automorphic extensions of finite simple
  exceptional groups of Lie type}, to appear in \emph{Algebra Logic}.

\bibitem{14Zve}
M.~A. Zvezdina, {Spectra of automorphic extensions of finite simple symplectic
  and orthogonal groups over fields of characteristic $2$}, \emph{Sib.
  {\'E}lektron. Mat. Izv.} \textbf{11} (2014), 823--832.

\end{thebibliography}
\end{document}